\documentclass[12pt, reqno]{article}
\usepackage[english]{babel}

\usepackage{amsmath}
\usepackage{amssymb}
\usepackage{amsthm}
\numberwithin{equation}{section}
\numberwithin{figure}{section}
\numberwithin{table}{section}
\usepackage{graphicx}
\usepackage[colorlinks=true, allcolors=black]{hyperref}
\usepackage{booktabs}
\usepackage{multirow}
	\usepackage{authblk}

\newtheorem{theorem}{Theorem}[section]
\newtheorem{remark}{Remark}[section]

\newtheorem{example}{Example}[section]

\begin{document}

\title{Well-balanced path-conservative discontinuous Galerkin methods with equilibrium preserving space for shallow water linearized moment equations}

\author[1]{\normalsize Ruilin Fan}
\author[2,3]{Julian Koellermeier}
\author[1,$*$]{Yinhua Xia}
\author[1,4]{Yan Xu}
\author[1]{Jiahui Zhang}
\affil[1]{School of Mathematical Sciences, University of Science and Technology of China}
\affil[2]{Bernoulli Institute, University of Groningen}
\affil[3]{Department of Mathematics, Computer Science and Statistics, Ghent University}
\affil[4]{Laoshan Laboratory}
\affil[$*$]{Corresponding author, Email: yhxia@ustc.edu.cn}
\renewcommand*{\Affilfont}{\small\it}  
	\renewcommand\Authands{ and }

\date{}
\maketitle

% \tableofcontents

\begin{abstract}
This paper presents high-order, well-balanced, path-conservative discontinuous Galerkin (DG) methods for the shallow water linearized moment equations (SWLME), designed to preserve both still and moving water equilibrium states.
Unlike the multi-layer shallow water equations, which model vertical velocity variations using multiple distinct layers, the SWLME employs a polynomial expansion of velocity profiles with up to $N$ moments. This approach enables a more detailed representation of vertical momentum transfer and complex velocity profiles while retaining hyperbolicity. However, the presence of non-conservative terms and complex steady-state structures introduces significant numerical challenges.
Addressing these challenges, we develop path-conservative DG schemes grounded in the Dal Maso-LeFloch-Murat (DLM) theory for non-conservative products. Our method balances flux gradients, non-conservative terms, and source terms through equilibrium-preserving spaces. For the still water equilibrium, we reformulate the equations into a quasilinear form that eliminates source terms, inherently preserving steady states. For the moving water equilibrium, we extend the DG method by transforming conservative variables into equilibrium variables and employing linear segment paths.
Theoretical analysis and numerical experiments demonstrate that the proposed methods achieve exact equilibrium preservation while maintaining high-order accuracy, even in scenarios with vertical velocity variations and complex topographies.

\textbf{Keywords:} {Shallow water linearized moment equations; discontinuous Galerkin method;  equilibrium preserving space; non-conservative products; path-conservative scheme;  well-balanced}

\end{abstract}

\section{Introduction}
The shallow water equations (SWE) are a class of hyperbolic partial differential equations that describe fluid motion beneath a pressure surface \cite{vreugdenhil2013numerical}. They are derived by integrating the Navier-Stokes equations over the vertical direction under the assumption that horizontal length scales are much larger than vertical ones. As a fundamental model for simulating free-surface flows, the SWE are widely used in hydrology \cite{schijf1953theoretical}, oceanography, coastal engineering, and environmental sciences \cite{courtier1988global}. A key assumption of the SWE is that the horizontal velocity is represented by its average, which is constant throughout the water column, from the bottom to the surface, which is generally sufficient for simple flow conditions \cite{toro2013riemann}. However, this simplification becomes inadequate in scenarios involving strong vertical velocity gradients, such as turbulent boundary layers, sediment transport, or flows with significant frictional effects \cite{koellermeier2020analysis}.

To address these limitations, extensions of the SWE have been developed. One such extension is the multi-layer shallow water model, which introduces multiple superposed layers of immiscible fluids to represent vertical velocity variation \cite{burger2025multilayer}. While this approach captures some aspects of vertical structure, it is inherently limited to a piecewise constant approximation and cannot resolve the full vertical velocity profile. To overcome this, the shallow water moment equations (SWME) were proposed in \cite{kowalski2019moment}, introducing a framework that expands the vertical velocity profile using polynomial basis functions. In the SWME model, vertical variations in velocity are captured by augmenting the system with $N$ additional variables, referred to as moments, which represent deviations from the depth-averaged velocity. This formulation allows the velocity to vary continuously with depth, enabling the model to resolve higher-order moments of the flow and more accurately simulate vertical momentum transfer and turbulence effects. This advancement could be particularly important in applications such as snow avalanches and granular flows, where vertical velocity variations play a critical role in the dynamics \cite{christen2010ramms, kowalski2013shallow}.

Despite the enhanced physical fidelity of the shallow water moment equations (SWME), the model introduces two principal challenges. The first is the loss of hyperbolicity. As analyzed in \cite{koellermeier2020analysis}, the original SWME model becomes non-hyperbolic for moment orders $N>1$, which is a significant limitation for robust and stable numerical simulations.
To address this issue, the hyperbolic shallow water moment equations (HSWME) and the related $\beta$-HSWME were proposed. These formulations restore hyperbolicity while preserving a level of accuracy comparable to the original SWME model. This model has also been successfully applied in various settings, including sediment transport \cite{garres2020shallow}. The second challenge lies in the analysis of steady states. Accurately capturing steady states is crucial in numerical simulations, as it ensures that solutions initialized at equilibrium remain stable and free from spurious oscillations or artifacts. However, the non-conservative structure of the SWMEs complicates the identification and preservation of steady-state solutions. As noted in \cite{koellermeier2022127166}, this difficulty persists even in the hyperbolic variants, HSWME and  $\beta$-HSWME, due to their intrinsic non-conservative terms, which hinder straightforward analysis and numerical treatment of equilibrium states.

% To overcome the limitations of the SWME and their hyperbolic variants, we adopt the Shallow Water Linearized Moment Equations (SWLME). The SWLME are derived by linearizing the original SWME system around a constant velocity profile \cite{koellermeier2022127166}.
To overcome the limitations of the SWME and their hyperbolic variants, we adopt the shallow water linearized moment equations (SWLME), which are derived by linearizing the original SWME system around a constant velocity profile, with the linearization applied only to the higher-order moment equations to retain momentum balance and complex steady states \cite{koellermeier2022127166}.
In this formulation, nonlinear contributions of the basis coefficients are neglected in the higher moment equations, while linear contributions from all coefficients are retained. This simplification decouples certain terms in the system while preserving the essential structure of the model. Importantly, the SWLME is provably hyperbolic, which ensures numerical stability, and its steady states are significantly easier to characterize compared to those of the SWME, HSWME, and $\beta$-HSWME. In particular, the steady states can be computed analytically for a flat bottom as an extension of the typical shallow water jumps, due to the reduced number of non-conservative terms and the simplified structure of the higher order moment equations.

Numerical experiments in \cite{koellermeier2022127166} highlight the advantages of the SWLME over existing formulations such as the HSWME and $\beta$-HSWME. While the SWLME exhibits stable and accurate results in these simulations, a recent analytical study \cite{huang2022equilibrium} shows that one of its equilibrium states can be unstable under certain conditions, in contrast to the other moment models. Nevertheless, the model has consistently demonstrated robust performance in practical scenarios, especially in complex flow regimes over varying topography \cite{koellermeier2022127166}.

Another significant challenge posed by the SWMEs and SWLME is the presence of non-conservative products in the governing equations. This feature invalidates the classical notion of weak solutions and renders the traditional Rankine-Hugoniot jump conditions inapplicable, as emphasized in the foundational work of Dal Maso, LeFloch, and Murat (DLM) \cite{maso1995}. To address this issue, the DLM theory introduces a generalized framework for weak solutions, wherein discontinuities are connected by a family of paths in the phase space. This path-dependent formulation allows for a rigorous definition of weak solutions in non-conservative systems. Building upon this theoretical foundation, Par\'es \cite{pares2006} developed the concept of path-conservative numerical schemes, which generalize classical conservative schemes to accommodate non-conservative hyperbolic systems. This methodology has been successfully applied in various contexts, as demonstrated in the works of Gosse \cite{gosselaurent2001} and Par\'es and Castro \cite{pares2004well}.
Nevertheless, it is crucial to recognize that the choice of the path (and its numerical discretization) plays a decisive role in the accuracy of the numerical solution. As shown by Abgrall et al. \cite{abgrall20102759}, even when the theoretically correct path is known, a path-conservative scheme may still yield inaccurate results if the implementation fails to capture the path-dependent structure accurately. This underscores the need for careful path selection and rigorous analysis when designing numerical methods for non-conservative systems.

%Well-balanced methods are a class of numerical schemes specifically designed to preserve steady-state solutions in the SWE and related models. Traditional numerical methods often introduce spurious oscillations or numerical artifacts when simulating steady flows, particularly in the presence of source terms. In contrast, well-balanced schemes aim to exactly preserve these steady states by ensuring a precise balance between the discretized numerical fluxes and source terms. This property is essential for accurately capturing small perturbations around steady flows and for maintaining numerical stability in long-term simulations.

Well-balanced methods are a class of numerical schemes specifically designed to preserve steady-state solutions in the SWE and related models. Traditional numerical methods often introduce spurious oscillations or numerical artifacts when simulating steady flows, particularly in the presence of source terms and near non-trivial moving-water equilibrium \cite{xing2011advantage}. In contrast, well-balanced schemes aim to exactly preserve these steady states by ensuring a precise balance between the discretized numerical fluxes and source terms. This property is essential for accurately capturing small perturbations around steady flows and for maintaining numerical stability in long-term simulations \cite{xing2014survey}.

Designing well-balanced schemes for the SWLME poses greater challenges than for the classical SWE. While the SWE is based on the assumption of a vertically uniform velocity profile, the SWLME introduces higher-order moment terms to account for vertical velocity variations. This results in a more complex system involving additional equilibrium variables and more intricate steady-state structures. These steady states include both the still water equilibrium (e.g., ``lake-at-rest") and the moving water equilibrium (e.g., subcritical or supercritical flow with non-zero moments). Achieving a well-balanced scheme for the SWLME requires maintaining an exact balance among flux gradients, non-conservative products, and source terms. This intricate interplay makes the development of well-balanced methods, particularly for moving water equilibrium, considerably more difficult than in the SWE case.

Recent research has actively addressed the challenges associated with designing well-balanced numerical schemes for the SWLME. Koellermeier and Pimentel-Garc\'ia \cite{koellermeier2022127166} introduced first- and second-order well-balanced finite volume schemes capable of preserving steady states numerically. Their results demonstrated the effectiveness of the model in capturing complex velocity profiles and varying bottom topographies. However, their methods were validated primarily up to second-order accuracy, leaving open the question of how to extend these schemes to higher-order formulations. More recently, Caballero-C\'ardenas et al. \cite{caballero2025117788} proposed a semi-implicit, second-order well-balanced method for the SWLME that employs a wave-splitting approach, separating acoustic and material waves to enhance computational efficiency in low Froude number regimes. The method achieves exact preservation of steady states by incorporating a reconstruction operator specifically tailored to the SWLME's structure. Nonetheless, this approach is inherently restricted to second-order accuracy and may not generalize easily to higher-order schemes.
In addition to accuracy constraints, the incorporation of friction effects further complicates numerical design.
As highlighted in Pimentel-Garc\'ia \cite{pimentel2024fully}, the inclusion of friction terms introduces stiffness into the system, making explicit schemes inefficient or unstable. Although implicit-explicit (IMEX) time integration techniques have been explored to mitigate stiffness, these approaches often rely on iterative solvers, which can significantly increase computational cost and implementation complexity.
An alternative approach is proposed in \cite{huang2022equilibrium}, where a first-order operator splitting method is employed: the friction term is treated in an implicit step that can be solved analytically, resulting in a fully explicit scheme overall. However, this method is limited to first-order accuracy.

This paper presents the first high-order well-balanced path-conservative discontinuous Galerkin (PCDG) methods for the SWLME. For the still water equilibrium, also known as the ``lake-at-rest" steady state, we redefine the conservative variables and reformulate the system into a quasilinear form that eliminates the explicit source term. This reformulation enables a natural balance between the flux gradients and the source terms, allowing the scheme to exactly preserve still water equilibrium without relying on hydrostatic reconstruction techniques \cite{audusse2004afa} or numerical flux modifications, even in the presence of smooth or discontinuous bottom topographies.
For moving water equilibrium, we extend the DG method to equilibrium-preserving spaces. This approach was initially developed for the classical shallow water equations by Zhang et al. \cite{zhang2023}, and subsequently generalized to other hyperbolic balance laws, such as the Ripa model and the Euler equations with gravitational forces in \cite{zhang2024}, and further extended to non-conservative hyperbolic systems in \cite{zhang2025}.
% Original: In this work, we adapt the methodology to the more complex SWLME, which involves a greater number of equilibrium variables and more intricate balance structures.
% modification version 1: In this work, we demonstrate the first application of the method to a system with an arbitrary number of equations, as represented by the more complex SWLME, which features a larger set of equilibrium variables and more intricate balance structures.
In this work, we adapt the methodology to the more complex SWLME, which has more equilibrium variables and a more complicated balance structure. This demonstrates the first application of the method to a system with a potentially arbitrary number of equations.
By transforming the conservative variables to equilibrium variables and employing linear segment paths within the path-conservative framework, we ensure that the resulting schemes maintain exact well-balanced properties.
Theoretical analysis and numerical experiments confirm that the proposed methods achieve high-order accuracy and preserve both still and moving water equilibria. Moreover, the DG framework, combined with suitable limiters, allows the scheme to robustly capture shocks and fine-scale flow features. Numerical results demonstrate the effectiveness and resolution capability of the proposed methods, even in challenging flow configurations involving discontinuities and complex topography.

% The rest of this paper is organized as follows. Section 2 describes the mathematical model of the SWLME and introduces the corresponding steady state solutions, along with the DLM theory for non-conservative products. Section 3 presents the path-conservative discontinuous Galerkin (PCDG) method designed to preserve the still water equilibrium for the SWLME. Section 4 extends the PCDG method to maintain the moving water equilibrium. Section 5 validates the proposed methods through extensive numerical experiments, demonstrating their well-balanced properties, high-order accuracy, and robustness in handling complex flow dynamics. Finally, Section 6 provides concluding remarks and outlines future research directions.
The rest of this paper is organized as follows. Section \ref{se-eq} describes the mathematical model of the SWLME and introduces the corresponding steady state solutions. Section \ref{se-sch} presents the path-conservative discontinuous Galerkin method designed to preserve the moving water and still water equilibrium for the SWLME. Section \ref{se-nu} provides numerical examples that demonstrate the well-balanced properties, high-order accuracy, and robustness in handling complex flow dynamics. Finally, Section \ref{se-con} provides concluding remarks and outlines future research directions.

\section{Shallow water linearized moment equations}\label{se-eq}
In this section, we will present the model of SWLME and its steady state solutions, together with a brief introduction to the DLM theory for the non-conservative products.

% \subsection{Shallow water linearized moment equations}
The one-dimensional SWLME (shallow water linearized moment equations) with $N$ moments, as introduced in \cite{koellermeier2022127166}, takes the form:
\begin{equation}\label{for:SWLME}
\partial_t\begin{bmatrix}h\\hu_m\\h\alpha_1\\\vdots\\h\alpha_N\end{bmatrix}
+\partial_x\begin{bmatrix}hu_m\\hu_m^2+\frac12gh^{2}+\frac13h\alpha_1^2+\dots+\frac1{2N+1}h\alpha_N^2\\2hu_m\alpha_1\\\vdots\\2hu_m\alpha_N\end{bmatrix}
=Q\partial_x\begin{bmatrix}h\\hu_m\\h\alpha_1\\\vdots\\h\alpha_N\end{bmatrix}
-\begin{bmatrix}0 \\ ghb_x \\ 0 \\ \vdots \\ 0\end{bmatrix},
\end{equation}
where $h$ is the water height, $u_m$ is the average horizontal velocity, and $\alpha_i,i=1,\dots,N$ are moment coefficients.
The horizontal velocity at depth $z$ can be expressed as
\begin{equation}\label{for:velocityExpan}
    u(\zeta)=u_m+\sum_{i=1}^N\alpha_i\phi_i(\zeta),\quad \phi_i(\zeta)=\frac{1}{i!}\frac{d^i}{d\zeta^i}(\zeta-\zeta^2)^i,
\end{equation}
where $\zeta$ is the scaled vertical coordinate $\zeta=(z-b)/h$.
The discharges are denoted by 
$$m_a=hu_m, m_i=h\alpha_i,i=1,\dots,N.$$ The matrix $Q=\operatorname{diag}(0,0,u_m,\dots,u_m)\in\mathbb{R}^{(N+2)\times(N+2)}$ represents the non-conservative coupling, $b$ denotes the bottom topography, and $g$ is the gravitational constant.

% For simplicity, we omit the effect of friction.
The conservative variables can be written as
\begin{equation}
    \boldsymbol{u}=(h,m_a,m_1,\dots,m_N)^T,
\end{equation}
where the superscript $T$ indicates the transpose. The system \eqref{for:SWLME} can be rewritten as a non-conservative balance law of the form
\begin{equation}\label{for:non-conBalanceLaw}
    \boldsymbol{u}_t+\boldsymbol{f}(\boldsymbol{u})_x+\mathcal{G}(\boldsymbol{u})\boldsymbol{u}_x=\boldsymbol{s}(\boldsymbol{u},b),
\end{equation}
where
\begin{equation}
    \boldsymbol{f}(\boldsymbol{u})=\begin{bmatrix}
    m_a\\hu_m^2+\frac{1}{2}gh^{2}+\sum\limits_{i=1}^N\frac{1}{2i+1}h\alpha_i^2\\2m_a \alpha_1\\\vdots\\2m_a \alpha_N
    \end{bmatrix},\,
    \mathcal{G}(\boldsymbol{u})=\begin{bmatrix}0& & & &\\&0& & &\\& &-u_m& &\\& & &\ddots&\\& & & &-u_m\end{bmatrix},\,
    \boldsymbol{s}(\boldsymbol{u},b)=\begin{bmatrix}
    0 \\ -ghb_x \\ 0 \\ \vdots \\ 0
    \end{bmatrix}.
\end{equation}

The Jacobian matrix $\mathcal{A}(\boldsymbol{u})$ of \eqref{for:SWLME} is given by
\begin{equation}\label{for:jacobianmat}
    \mathcal{A}(\boldsymbol{u}) = \frac{\partial\boldsymbol{f}}{\partial\boldsymbol{u}}+\mathcal{G}(\boldsymbol{u}) =
    \begin{bmatrix}
    0 & 1 & 0 & \cdots & 0 \\
    gh-u_{m}^{2}-\frac{\alpha_{1}^{2}}{3}-\dots-\frac{\alpha_{N}^{2}}{2N+1} & 2u_{m} & \frac{2\alpha_{1}}{3} & \cdots & \frac{2\alpha_{N}}{2N+1} \\
    -2u_{m}\alpha_{1} & 2\alpha_{1} & u_{m} & & \\
    \vdots & \vdots & & \ddots & \\
    -2u_{m}\alpha_{N} & 2\alpha_{N} & & & u_{m}
    \end{bmatrix}.
\end{equation}
According to Theorem 1 in \cite{koellermeier2022127166}, the characteristic polynomial of the matrix $\mathcal{A}(\boldsymbol{u})$ is
\begin{equation}
    P(\lambda)=(u_m-\lambda)\left[(\lambda-u_m)^2-gh-\sum_{i=1}^N\frac{3\alpha_i^2}{2i+1}\right],
\end{equation}
which leads to the eigenvalues describing the propagation speeds of the system
\begin{equation}\label{for:eigen}
    \lambda_{1,2}=u_{m}\pm\sqrt{gh+\sum_{i=1}^{N}\frac{3\alpha_{i}^{2}}{2i+1}},\quad\lambda_{i+2}=u_{m},\quad\mathrm{for~}i=1,\dots,N.
\end{equation}
The full set of independent eigenvectors can be found in \cite{koellermeier2022127166}. Therefore, the system is hyperbolic whenever the water height is positive.

The system \eqref{for:SWLME} admits non-trivial steady-state solutions, in which the sum of the non-zero flux gradient and the non-conservative product exactly balances the source term. The moving water equilibrium state solutions are given by
\begin{equation}\label{for:moving}
    \begin{aligned}
    hu_{m} &= \text{constant},\\
    E:=\frac{1}{2}u_{m}^{2}+g(h+b)+\frac{3}{2}\sum_{i=1}^{N}\frac{1}{2i+1}\alpha_{i}^{2}&=\text{constant},\\
    \frac{\alpha_{i}}{h} &= \text{constant}\quad\text{for~} i=1,\dots,N,\end{aligned}
\end{equation}
where $E$  denotes the energy \cite{caballero2025117788}.

As a special case of \eqref{for:moving}, when $u_m=0$ and $\alpha_i=0$, according to \eqref{for:velocityExpan}, we obtain $u=0$, i.e., the still water equilibrium (also known as the ``lake at rest" state). At this point, SWLME reduces to standard SWE
\begin{equation}\label{for:still}
    \begin{aligned}
    & h+b = \text{constant}, \\
    & u_m=0, \quad \alpha_i=0\quad\text{for~} i=1,\dots,N.
    \end{aligned}
\end{equation}

To handle the non-conservative products in the SWLME system \eqref{for:SWLME}, we adopt the Dal Maso-LeFloch-Murat (DLM) theory \cite{maso1995}, which provides a rigorous framework for defining weak solutions of non-conservative hyperbolic systems. This theory introduces a family of paths connecting the left and right states across discontinuities, thereby ensuring that the non-conservative terms are well-defined and the resulting weak solutions are stable.
% Consider the following problem
% \begin{equation}
%     \frac{\partial\boldsymbol{u}}{\partial t}+\mathcal{A}(\boldsymbol{u})\frac{\partial\boldsymbol{u}}{\partial x}=0,\quad x\in\mathbb{R},\quad t>0,
% \end{equation}
% where $\Omega$ is an open convex subset of $\mathbb{R}^N$, $\boldsymbol{u}(x,t)\in\Omega$ and $\boldsymbol{u}\mapsto\mathcal{A}(\boldsymbol{u})\in\mathbb{R}^{N\times N}$ is a smooth locally bounded map.

At a discontinuity located at $x_0$, the non-conservative product is interpreted as a Borel measure $\mu$, defined by:
\begin{equation}
    \mu(\{x_0\}) = \int_0^1 \mathcal{G}(\Phi(s; u_L, u_R)) \frac{\partial \Phi}{\partial s}(s; u_L, u_R) \, ds,
\end{equation}
where $\mathcal{G}(\boldsymbol{u})$ is the non-conservative matrix in the non-conservative balance law \eqref{for:non-conBalanceLaw}, and $\Phi(s; u_L, u_R)$ is a Lipschitz continuous path function connecting the left state $u_L$ and the right state $u_R$, with $s\in [0,1]$ denoting the path parameter. By incorporating DLM theory into our numerical schemes, we can retain high-order accuracy while ensuring the well-balanced preservation of steady-state solutions in the SWLME system.

\section{Well-balanced DG methods with equilibrium preserving space}
\label{se-sch}
In this section, we introduce the equilibrium preserving space and equilibrium variables to construct a well-balanced scheme. These elements are designed to ensure an exact balance among the flux gradients, non-conservative products, and source terms, which are crucial for maintaining the well-balanced property. We first develop a path-conservative discontinuous Galerkin (PCDG) scheme tailored for the moving water equilibrium, which incorporates specialized treatments of both the numerical flux and the path integral to preserve equilibrium states. The corresponding PCDG scheme for the still water equilibrium, representing a simpler case, will be presented thereafter.

\subsection{Notations}
The computational domain $\Omega$ is divided into $nx$ non-overlapping subintervals, and we have
\begin{equation}
    \Omega=\bigcup_{j=1}^{nx}\mathcal{I}_j=\bigcup_{j=1}^{nx}[x_{j-\frac{1}{2}},x_{j+\frac{1}{2}}],
\end{equation}
where $x_{j\pm\frac{1}{2}}$ are the left and right cell interfaces of cell $\mathcal{I}_j$, respectively. The finite-dimensional test function space is defined by
\begin{equation}
    \mathcal{W}_h^k:=\{w:w|_{\mathcal{I}_j}\in P^k(\mathcal{I}_j),\forall  j = 1, \dots, nx \},
\end{equation}
where $P^k(\mathcal{I}_j)$ denotes the space of polynomials of degree at most $k$ on cell $\mathcal{I}_j$. Accordingly, the vector-valued finite element space with $N$ components  is defined as
\begin{equation}
    \mathcal{V}_h^k:=\left\{\boldsymbol{v}:\boldsymbol{v}=(v_1,\dots,v_{N})^T|v_1,\dots,v_{N}\in\mathcal{W}_h^k\right\}.
\end{equation}
In the DG framework, any unknown scalar function $u$ is approximated within the space $\mathcal{W}_h^k$. For simplicity, the discrete approximation is denoted by the same symbol $u$. At a cell interface  $x_{j+\frac{1}{2}}$, the one-sided limits from the left and right are denoted by
\begin{equation}
    u_{j+\frac{1}{2}}^-:=\lim_{\epsilon\to0^+}u(x-\epsilon),u_{j+\frac{1}{2}}^+:=\lim_{\epsilon\to0^+}u(x+\epsilon).
\end{equation}

\subsection{Conservative and equilibrium variables}
To preserve the moving water equilibrium \eqref{for:moving}, we adopt a strategy commonly used in path-conservative methods (see, e.g., \cite{castrodiaz2013, pares2008, munozruiz2011}) by introducing an auxiliary equation that is trivially satisfied
\begin{equation}
    b_t=0,
\end{equation}
reflecting the fact that the bottom topography $b(x)$ is time-independent. We then redefine the vector of conservative variables to include the bottom topography
\begin{equation}\label{var:cons}
    \boldsymbol{u}=(h,m_a,m_1,\dots,m_N,b)^T.
\end{equation}
With this extended system, equation \eqref{for:SWLME} can be rewritten in quasilinear form as
\begin{equation}\label{for:quasilin}
    \boldsymbol{u}_t+\boldsymbol{f}(\boldsymbol{u})_x+\mathcal{G}(\boldsymbol{u})\boldsymbol{u}_x=0,
\end{equation}
where
\begin{equation}\label{for:FandG}
    \boldsymbol{f}(\boldsymbol{u})=\begin{bmatrix}m_a\\hu_{m}^{2}+\frac{1}{2}gh^{2}+\sum\limits_{i=1}^N\frac{1}{2i+1}h\alpha_N^2\\2m_a \alpha_1\\\vdots\\2m_a \alpha_N\\0\end{bmatrix},\quad
    \mathcal{G}(\boldsymbol{u})=\begin{bmatrix}0& & & & &\\&0& & & &gh\\& &-u_{m}& & &\\& & &\ddots& &\\& & & &-u_{m}&\\& & & & &0\end{bmatrix}.
\end{equation}

To preserve the moving water equilibrium, we follow the approach proposed in \cite{zhang2025} by introducing the equilibrium variables
\begin{equation}\label{var:equ}
    \boldsymbol{v} = (E,hu_m,\frac{\alpha_1}{h},\dots,\frac{\alpha_N}{h},b)^T.
\end{equation}
The essential part of developing our well-balanced path-conservative DG method is the transformation between the conservative variables $\boldsymbol{u}$ and the equilibrium variables $\boldsymbol{v}$ defined in \eqref{for:moving}. On the one hand, the equilibrium variables can be calculated directly from the conservative variables as
\begin{equation}\label{for:equ2cons}
    \boldsymbol{v}=\boldsymbol{v}(\boldsymbol{u})=
    \begin{bmatrix}
    E\\hu_m\\\frac{\alpha_1}{h}\\\vdots\\\frac{\alpha_N}{h}\\b
    \end{bmatrix}=
    \begin{bmatrix}
    \frac{1}{2}\frac{(hu_m)^2}{h^2}+g(h+b)+\frac{3}{2}\sum_{i=1}^{N}\frac{1}{2i+1}\frac{(h\alpha_i)^2}{h^2}\\
    hu_m\\ \frac{h\alpha_1}{h^2} \\ \vdots \\ \frac{h\alpha_N}{h^2}\\b
    \end{bmatrix}.
\end{equation}
On the other hand, the inverse transformation from $\boldsymbol{v}$ to $\boldsymbol{u}$ is less straightforward due to the nonlinear dependence of the energy $E$ on the water height $h$. In particular, recovering $h$ requires solving the following quartic polynomial equation
\begin{equation}\label{for:quartic}
    q(h) = \left[\frac{3}{2}\sum_{i=1}^{N}\frac{1}{2i+1}(\frac{\alpha_{i}}{h})^2\right]h^4 + gh^3 + (gb-E)h^2 + \frac12(hu_m)^2 = 0.
\end{equation}
If all moment coefficients $\alpha_{i}=0$, this reduces to a cubic equation that can be solved analytically using the Froude number, cf.\ \cite{zhang2023}. Otherwise, the quartic equation must be solved numerically using a nonlinear iterative method, such as Newton's iteration.
Once $h$ is recovered, the remaining conservative variables are computed directly according to \eqref{for:equ2cons}.
% \begin{equation}
%     hu_m = hu_m,\quad h\alpha_i = \frac{\alpha_i}{h}h^2.
% \end{equation}

The flux terms in the non-conservative balance law \eqref{for:quasilin} can be rewritten in terms of the equilibrium variables as
\begin{equation}\label{for:FandG2}
    \boldsymbol{f}(\boldsymbol{u})_x+\mathcal{G}(\boldsymbol{u})\boldsymbol{u}_x=\mathcal{L}(\boldsymbol{u})\widetilde{\boldsymbol{v}}(\boldsymbol{u})_x,
\end{equation}
 where $\mathcal{L}(\boldsymbol{u})$ is a matrix depending on the conservative variables, and $\widetilde{\boldsymbol{v}}(\boldsymbol{u})$ is a modified vector of equilibrium variables in which the bottom topography component is set to zero
\begin{equation}
    \widetilde{\boldsymbol{v}}=(E,hu_m,\frac{\alpha_1}{h},\dots,\frac{\alpha_N}{h},0)^T.
\end{equation}

Here, we consider a simple scenario where the equilibrium variables $\widetilde{\boldsymbol{v}}$ are defined through algebraic relations. It can be verified that the moving water equilibrium state \eqref{for:moving} satisfies identity \eqref{for:FandG2}, confirming that this formulation is consistent with the steady-state conditions. The matrix $\mathcal{L}(\boldsymbol{u})$ can be determined via the method of undetermined coefficients, yielding the following explicit expression
\begin{equation}
    \mathcal{L}(\boldsymbol{u}) =
    \begin{bmatrix}
    0 & 1 & 0 & \cdots & 0 & 0\\
    h & u & -\frac{1}{3}h^2\alpha_1 & \cdots & -\frac{1}{2N+1}h^2\alpha_{N} & 0\\
    0 & 2\alpha_1 & h^2u_{m} & \cdots & 0 & \vdots\\
    \vdots & \vdots & \vdots & \ddots & \vdots & \vdots\\
    0 & 2\alpha_{N} & 0 & \cdots & h^2u_{m} & 0\\
    0 & 0 & 0 & \cdots & 0 & 1
\end{bmatrix}.
\end{equation}
This formulation facilitates the numerical computation of the path integral term by simplifying its structure and ensuring compatibility with equilibrium preservation.

\subsection{The semi-discrete well-balanced scheme}
To establish a numerical scheme that preserves the moving water equilibrium, we perform the computation within the equilibrium preserving space. That is, we seek the approximation of the solution for the equilibrium variables $\boldsymbol{v}$ in the DG finite element space, rather than the approximation of the solution for the conservative variables $\boldsymbol{u}$. Consequently, we can rewrite the system \eqref{for:quasilin} in terms of the equilibrium variables $\boldsymbol{v}$ \eqref{var:equ}
\begin{equation}
    \boldsymbol{u}(\boldsymbol{v})_t+\boldsymbol{f}(\boldsymbol{u}(\boldsymbol{v}))_x+\mathcal{G}(\boldsymbol{u}(\boldsymbol{v}))\boldsymbol{u}(\boldsymbol{v})_x=0,
\end{equation}
where $\boldsymbol{f}(\boldsymbol{u})$ and $\mathcal{G}(\boldsymbol{u})$ are the flux term and non-conservative matrix in \eqref{for:FandG}.

This approach treats the conservative variables as nonlinear functions of the equilibrium variables. While the form of the governing equations remains unchanged, all computations are carried out in terms of the equilibrium variables. This constitutes the primary distinction between our method and standard DG schemes.

The well-balanced PCDG scheme for the moving water equilibrium is defined as: Find $\boldsymbol{v}\in\mathcal{V}_h^k$, s.t. for any test function $\boldsymbol{\varphi}\in\mathcal{V}_h^k$ we have
\begin{equation}\label{for:semi-PCDG2}
    \frac{d}{dt}\int_{\mathcal{I}_j}\boldsymbol{u}(\boldsymbol{v})\cdot\boldsymbol{\varphi}dx=\mathrm{RHS}_j^M(\boldsymbol{v},\boldsymbol{\varphi})
\end{equation}
with
\begin{equation}\label{for:RHS2}
    \begin{aligned}
    \mathrm{RHS}_{j}^{M}(\boldsymbol{v},\boldsymbol{\varphi}) & =\int_{\mathcal{I}_{j}}\boldsymbol{f}(\boldsymbol{u}(\boldsymbol{v}))\cdot\boldsymbol{\varphi}_{x}dx-\widehat{\boldsymbol{f}}_{j+\frac{1}{2}}^{\mathrm{mod}}\cdot\boldsymbol{\varphi}_{j+\frac{1}{2}}^{-}+\widehat{\boldsymbol{f}}_{j-\frac{1}{2}}^{\mathrm{mod}}\cdot\boldsymbol{\varphi}_{j-\frac{1}{2}}^{+} \\
    & -\int_{\mathcal{I}_{i}}\mathcal{G}(\boldsymbol{u}(\boldsymbol{v}))\boldsymbol{u}(\boldsymbol{v})_{x}\cdot\boldsymbol{\varphi}dx-\frac{1}{2}\boldsymbol{\varphi}_{j+\frac{1}{2}}^{-}\cdot g_{\Phi,j+\frac{1}{2}}-\frac{1}{2}\boldsymbol{\varphi}_{j-\frac{1}{2}}^{+}\cdot g_{\Phi,j-\frac{1}{2}}.
\end{aligned}
\end{equation}
% where $\Phi_{j+\frac{1}{2}}(s)$ is a Lipschitz continuous path function defined for moving equilibrium preserving PCDG scheme, and $\widehat{\boldsymbol{f}}_{j+\frac{1}{2}}^{\mathrm{mod}}$ is the modified numerical flux.
Here, $\Phi_{j+\frac{1}{2}}(s)$ is a Lipschitz continuous path function defined for the moving equilibrium preserving PCDG scheme, and $\widehat{\boldsymbol{f}}_{j+\frac{1}{2}}^{\mathrm{mod}}$ is the modified numerical flux. The definitions of the numerical flux, the path function $\Phi_{j+\frac{1}{2}}(s)$, and the path integral term $g_{\Phi,j\pm\frac{1}{2}}$   will be provided in section \ref{se-flux}-\ref{se-path}.

\subsubsection{Special case for still water}
As the still water case is a special case of the moving water equilibrium, the moving water equilibrium-preserving PCDG scheme inherently satisfies the still water equilibrium condition. However, in the still water case, the scheme can be further simplified since there is no need to compute the equilibrium variables, leading to improved computational efficiency as explained in the following.  

Motivated by the approach proposed in \cite{zhang2025}, we define the water surface elevation as $H = h + b$. Using this, the equilibrium variables and the corresponding still water equilibrium state can be written as
\begin{equation}
    \boldsymbol{w}=(H,m_a,m_1,\dots,m_N)^T = (\text{constant},0,0,\dots,0)^T.
\end{equation}
Note that in this case, there is no nonlinear relationship between the equilibrium variables and the conservative variables. The transformation simply replaces $h$ with $H$. To distinguish this setting, we denote the equilibrium variables for the still water equilibrium-preserving scheme by $\boldsymbol{w}$. Accordingly, the system \eqref{for:SWLME} is reformulated in the following  form
\begin{equation}
    \boldsymbol{w}_t+\boldsymbol{f}(\boldsymbol{w})_x+\mathcal{G}(\boldsymbol{w})\boldsymbol{w}_x=0.
\end{equation}
The flux term and the non-conservative matrix are given by
\begin{equation}\label{equ:SWLMEstill}
\boldsymbol{f}(\boldsymbol{w})=
\begin{bmatrix}
m_a\\\frac{m_a^2}{H-b}+\frac12gH^{2}+\sum_{i=1}^N\frac1{2i+1}\frac{m_i^2}{H-b}\\2m_a\frac{m_1}{H-b}\\\vdots\\2m_a\frac{m_N}{H-b}
\end{bmatrix},
\quad\mathcal{G}(\boldsymbol{w})=
\begin{bmatrix}
    0 & & & & \\-gb & 0 & & & \\ & &-\frac{m_a}{H-b}& &\\ & & &\ddots& \\ & & & &-\frac{m_a}{H-b}
\end{bmatrix}.
\end{equation}
This approach avoids the influence of discontinuities in the water height $h$ on the flux terms, thereby achieving still water equilibrium preservation without requiring additional modifications to the numerical flux.

Following the framework introduced in \cite{rhebergen2008}, the well-balanced PCDG scheme for still water equilibrium is defined as follows: Find $\boldsymbol{w}\in\mathcal{V}_h^k$, s.t. for any test function $\boldsymbol{\varphi}\in\mathcal{V}_h^k$ we have
\begin{equation}\label{for:semi-PCDG1}
    \frac{d}{dt}\int_{\mathcal{I}_j}\boldsymbol{w}\cdot\boldsymbol{\varphi}dx=\mathrm{RHS}_j^S(\boldsymbol{w},b,\boldsymbol{\varphi})
\end{equation}
with
\begin{equation}\label{for:RHS1}
\begin{aligned}
    \mathrm{RHS}_{j}^{S}(\boldsymbol{w},b,\boldsymbol{\varphi}) & =\int_{\mathcal{I}_j}\boldsymbol{f}(\boldsymbol{w})\cdot\boldsymbol{\varphi}_xdx-\widehat{\boldsymbol{f}}_{j+\frac{1}{2}}\cdot\boldsymbol{\varphi}_{j+\frac{1}{2}}^-+\widehat{\boldsymbol{f}}_{j-\frac{1}{2}}\cdot\boldsymbol{\varphi}_{j-\frac{1}{2}}^+ \\
    & -\int_{\mathcal{I}_{j}}\mathcal{G}(\boldsymbol{w})\boldsymbol{w}_{x}\cdot\boldsymbol{\varphi}dx-\frac{1}{2}\boldsymbol{\varphi}_{j+\frac{1}{2}}^{-}\cdot g_{\Phi,j+\frac{1}{2}}-\frac{1}{2}\boldsymbol{\varphi}_{j-\frac{1}{2}}^{+}\cdot g_{\Phi,j-\frac{1}{2}}.
\end{aligned}
\end{equation}
To ensure the stability of the scheme, we use the monotone Lax-Friedrichs numerical flux, defined as
\begin{equation}\label{for:LF-flux}
    \widehat{\boldsymbol{f}}_{j+\frac{1}{2}}=\widehat{\boldsymbol{f}}(\boldsymbol{w}_{j+\frac{1}{2}}^-,\boldsymbol{w}_{j+\frac{1}{2}}^+)=\frac{1}{2}\left(\boldsymbol{f}(\boldsymbol{w}_{j+\frac{1}{2}}^-)+\boldsymbol{f}(\boldsymbol{w}_{j+\frac{1}{2}}^+)\right)-\frac{\alpha}{2}(\boldsymbol{w}_{j+\frac{1}{2}}^+-\boldsymbol{w}_{j+\frac{1}{2}}^-),
\end{equation}
and
\begin{equation}
    \alpha=\max_{\boldsymbol{w}}\{|\lambda_1(\boldsymbol{w})|,\dots,|\lambda_{N+3}(\boldsymbol{w})|\},
\end{equation}
where $\lambda_i(\boldsymbol{w})$ are the eigenvalues of the Jacobi matrix $\mathcal{A}$ in \eqref{for:jacobianmat} and eigenvalues are given in \eqref{for:eigen}. The maximum is taken globally over the computational domain. The definition and computation of the path integral term $g_{\Phi,j\pm\frac{1}{2}}$ is provided in section \ref{se-path}.

%The definition and computation of the path integral term $g_{\Phi,j\pm\frac{1}{2}}$ will be provided in the next section.

%\subsection{Well-balanced treatment}
%To preserve a wider range of moving water equilibrium steady state, several additional techniques are employed here.

\subsubsection{Flux modification}\label{se-flux}
For the moving water equilibrium-preserving scheme, we follow the approach presented in \cite{zhang2025} and modify the Lax-Friedrichs numerical flux using the hydrostatic reconstruction technique, originally proposed in \cite{audusse2004afa}. The modified numerical flux is defined as
\begin{equation}\label{for:LF-fluxmod}
    \widehat{\boldsymbol{f}}_{j+\frac{1}{2}}^{\mathrm{mod}}=\widehat{\boldsymbol{f}}^{\mathrm{mod}}(\boldsymbol{u}_{j+\frac{1}{2}}^-,\boldsymbol{u}_{j+\frac{1}{2}}^+)=\frac{1}{2}\left(\boldsymbol{f}(\boldsymbol{u}_{j+\frac{1}{2}}^-)+\boldsymbol{f}(\boldsymbol{u}_{j+\frac{1}{2}}^+)\right)-\frac{\alpha}{2}(\boldsymbol{u}_{j+\frac{1}{2}}^{*,+}-\boldsymbol{u}_{j+\frac{1}{2}}^{*,-}),
\end{equation}
and
\begin{equation}
    \alpha=\max_{\boldsymbol{u}}\{|\lambda_1(\boldsymbol{u})|,\dots,|\lambda_{N+3}(\boldsymbol{u})|\},
\end{equation}
with $\lambda_i(\boldsymbol{u})$ the eigenvalues of the Jacobi matrix $\mathcal{A}$ in \eqref{for:jacobianmat} and the maximum is taken
globally over the computational domain. The modified cell interface values $\boldsymbol{u}_{j+\frac{1}{2}}^{*,\pm}$ are defined as
\begin{equation}\label{for:umod}
    \boldsymbol{u}_{j+\frac{1}{2}}^{*,\pm}=(h_{j+\frac{1}{2}}^{*,\pm},(hu_m)_{j+\frac{1}{2}}^{*,\pm},(h\alpha_1)_{j+\frac{1}{2}}^{*,\pm},\dots,(h\alpha_N)_{j+\frac{1}{2}}^{*,\pm},b_{j+\frac{1}{2}}^{*})^T,
\end{equation}
where the reconstructed water height $h_{j+\frac{1}{2}}^{*,\pm}$ is obtained by solving the following quartic equation numerically
\begin{equation}\label{for:quartic2}
    \left[\frac{3}{2}\sum_{i=1}^{N}\frac{1}{2i+1}\left((\frac{\alpha_{i}}{h})_{j+\frac{1}{2}}^{\pm}\right)^2\right](h_{j+\frac{1}{2}}^{*,\pm})^4 + g(h_{j+\frac{1}{2}}^{*,\pm})^3 + (gb_{j+\frac{1}{2}}^{*}-E_{j+\frac{1}{2}}^{\pm})(h_{j+\frac{1}{2}}^{*,\pm})^2 + \frac12\left((hu_m)_{j+\frac{1}{2}}^{\pm}\right)^2 = 0.
\end{equation}
The interface bottom topography value $b_{j+\frac{1}{2}}^*$ is taken as the minimum of the two neighboring values to preserve the well-balanced property
\begin{equation}
    b_{j+\frac{1}{2}}^*=\min{(b_{j+\frac{1}{2}}^-,b_{j+\frac{1}{2}}^+)}.
\end{equation}

\subsubsection{Path integral term calculation}\label{se-path}
For both the moving water and still water equilibrium-preserving schemes, we adopt the definition of non-conservative products from \cite{maso1995}. Specifically, we define $\Phi_{j+\frac{1}{2}}(s)$ as a Lipschitz continuous path function, where a linear segment path connects the left and right states across the discontinuity. The precise definition is given by
\begin{equation}\label{for:linpath1}
    \Phi_{j+\frac{1}{2}}(s):=\Phi(s;\boldsymbol{u}_{j+\frac{1}{2}}^{-},\boldsymbol{u}_{j+\frac{1}{2}}^{+})=\boldsymbol{u}_{j+\frac{1}{2}}^{-}+s(\boldsymbol{u}_{j+\frac{1}{2}}^{+}-\boldsymbol{u}_{j+\frac{1}{2}}^{-}),\quad s\in[0,1].
\end{equation}
The path integral term $g_{\Phi,j\pm\frac{1}{2}}$ in \eqref{for:RHS2} and \eqref{for:RHS1}, which represents the contributions of the non-conservative products, is given by
\begin{equation}\label{for:pathinter}
    g_{\Phi,j+\frac{1}{2}}=\int_0^1\mathcal{G}(\Phi_{j+\frac{1}{2}}(s))\frac{\partial}{\partial s}\Phi_{j+\frac{1}{2}}(s)ds.
\end{equation}

% The path integral term $g_{\Phi^{P},j+\frac{1}{2}}$ is defined as
% \begin{equation}\label{for:pathinter}
%     g_{\Phi^P,j+\frac{1}{2}}=\int_0^1\mathcal{G}(\Phi_{j+\frac{1}{2}}^P(s))\frac{\partial}{\partial s}\Phi_{j+\frac{1}{2}}^P(s)ds,
% \end{equation}
% where $\Phi_{j+\frac{1}{2}}^P(s)$ is the path function. Here we also use the linear segment path to connect the left and right equilibrium states
% \begin{equation}\label{for:linpath2}
%     \Phi_{j+\frac{1}{2}}(s):=\Phi^{P}(s;\boldsymbol{v}_{j+\frac{1}{2}}^{-},\boldsymbol{v}_{j+\frac{1}{2}}^{+})=\boldsymbol{v}_{j+\frac{1}{2}}^{-}+s(\boldsymbol{v}_{j+\frac{1}{2}}^{+}-\boldsymbol{v}_{j+\frac{1}{2}}^{-}).
% \end{equation}
For the still water equilibrium preserving scheme, each component of $g_{\Phi,j+\frac{1}{2}}$ can be computed explicitly, owing to the linear segment path.
\begin{equation}\label{for:integralCal1}
    \begin{aligned}
    g_{\Phi,j+\frac{1}{2}}^{[2]} & =\int_{0}^{1}g\left((b)_{j+\frac{1}{2}}^{-}+s\left(b_{j+\frac{1}{2}}^{+}-b_{j+\frac{1}{2}}^{-}\right)\right)(H_{j+\frac{1}{2}}^{+}-H_{j+\frac{1}{2}}^{-})ds \\
    & =\frac{1}{2}g\left(b_{j+\frac{1}{2}}^{-}+b_{j+\frac{1}{2}}^{+}\right)(H_{j+\frac{1}{2}}^{+}-H_{j+\frac{1}{2}}^{-}),
\end{aligned}
\end{equation}
\begin{equation}\label{for:integralCal2}
    \begin{aligned}
    g_{\Phi,j+\frac{1}{2}}^{[i+2]} & =- \int_{0}^{1}\left((u_m)_{j+\frac{1}{2}}^{-}+s\left((u_m)_{j+\frac{1}{2}}^{+}-(u_m)_{j+\frac{1}{2}}^{-}\right)\right)\left((m_i) _{j+\frac{1}{2}}^{+}-(m_i)_{j+\frac{1}{2}}^{-}\right)ds \\
    & =- \left((u_m)_{j+\frac{1}{2}}^{-}+(u_m)_{j+\frac{1}{2}}^{+}\right)\left((m_i)_{j+\frac{1}{2}}^{+}-(m_i)_{j+\frac{1}{2}}^{-}\right),\quad i=1,\dots,N,
\end{aligned}
\end{equation}
where $$(u_m)_{j+\frac{1}{2}}^{\pm} = \frac{(m_a)_{j+\frac{1}{2}}^{\pm}}{H_{j+\frac{1}{2}}^{\pm}-b_{j+\frac{1}{2}}^{\pm}}, \; g_{\Phi,j+\frac{1}{2}}^{[1]} = 0.$$
Other types of paths can also be used but typically require numerical integration, while our approach with a linear path allows for exact computation of the path integration and does not involve any quadrature formula.

For the moving water equilibrium preserving scheme, we adopt the approach in \cite{KURGANOV2023, zhang2025} to develop a well-balanced formulation for computing this term. Notably, at the cell interface $x_{j+\frac{1}{2}}$, the relation \eqref{for:FandG2} holds, and we obtain
\begin{equation}\label{for:quasilin2}
    \frac{\partial}{\partial s}\boldsymbol{f}(\Phi_{j+\frac{1}{2}}(s))+\mathcal{G}(\Phi_{j+\frac{1}{2}}(s))\frac{\partial}{\partial s}\Phi_{j+\frac{1}{2}}(s)=\mathcal{L}(\Phi_{j+\frac{1}{2}}(s))\frac{\partial}{\partial s}\widetilde{\boldsymbol{v}}(\Phi_{j+\frac{1}{2}}(s)).
\end{equation}
Substituting \eqref{for:quasilin2} into the integral term \eqref{for:pathinter} yields
\begin{equation}\label{for:gpath}
    \begin{aligned}
    g_{\Phi,j+\frac12} &= \int_{0}^{1}{\mathcal{G}}(\Phi_{j+\frac{1}{2}}(s))\frac{\partial}{\partial s}\Phi_{j+\frac{1}{2}}(s)~ds \\
    &= \int_{0}^{1}\mathcal{L}(\Phi_{j+\frac{1}{2}}(s))\frac{\partial}{\partial s}\widetilde{\boldsymbol{v}}(\Phi_{j+\frac{1}{2}}(s)) ds-\int_{0}^{1}\frac{\partial}{\partial s}\boldsymbol{f}(\Phi_{j+\frac{1}{2}}(s)) ds \\
    &= \int_{0}^{1}\mathcal{L}(\Phi_{j+\frac{1}{2}}(s))\frac{\partial}{\partial s}\widetilde{\boldsymbol{v}}(\Phi_{j+\frac{1}{2}}(s)) ds-\boldsymbol{f}(\boldsymbol{u}_{j+\frac{1}{2}}^{+})+\boldsymbol{f}(\boldsymbol{u}_{j+\frac{1}{2}}^{-}) \\
    &= \int_{0}^{1}\mathcal{L}(\Phi_{j+\frac{1}{2}}(s))\left(\widetilde{\boldsymbol{v}}_{j+\frac{1}{2}}^{+}-\widetilde{\boldsymbol{v}}_{j+\frac{1}{2}}^{-}\right) ds -\boldsymbol{f}(\boldsymbol{u}_{j+\frac{1}{2}}^{+})+\boldsymbol{f}(\boldsymbol{u}_{j+\frac{1}{2}}^{-}).
\end{aligned}
\end{equation}
Analogous to \eqref{for:integralCal1} and \eqref{for:integralCal2}, the computation can be performed explicitly and no quadrature formula is needed. For brevity, we omit the details.

\subsubsection{Coupling term integral calculation}
To construct a well-balanced scheme, another essential requirement is the precise evaluation of the non-conservative product integral
\begin{equation}
    \int_{\mathcal{I}_j}\mathcal{G}(\boldsymbol{u}(\boldsymbol{v}))\boldsymbol{u}(\boldsymbol{v})_x\cdot\boldsymbol{\varphi}dx.
\end{equation}
For the still water equilibrium, this integral can be computed directly, as the calculation does not involve the nonlinear relationship between the equilibrium and conservative variables. However, the moving water equilibrium scheme involves the coupling terms
\begin{equation}
    \int_{\mathcal{I}_j}gh(\boldsymbol{v})b_x, \quad \int_{\mathcal{I}_j}-u_m(\boldsymbol{v})\left(h(\boldsymbol{v})\alpha_i(\boldsymbol{v})\right)_x.
\end{equation}
Given that our computational variables are the equilibrium variables $\boldsymbol{v}$, the conservative variables are non-polynomial functions of these equilibrium variables. An explicit expression for $\left(h(v)\alpha_i(v)\right)_x$ is not readily available. Instead, we apply the chain rule to express this term in terms of the derivatives of the equilibrium variables
\begin{equation}
    (h \alpha_i)_x = h (\alpha_i)_x + \alpha_i h_x,
\end{equation}
\begin{equation}
    (\alpha_i)_x = \left(\frac{\alpha_i}{h}h\right)_x = \left(\frac{\alpha_i}{h}\right)_x h + \frac{\alpha_i}{h} h_x,
\end{equation}
\begin{equation}
    h_x = \frac{\partial h}{\partial\boldsymbol{v}}\cdot \boldsymbol{v}_x
    = \frac{\partial h}{\partial E} E_x + \frac{\partial h}{\partial hu_m} (hu_m)_x + \sum_{i=1}^N \frac{\partial h}{\partial (\frac{\alpha_i}{h})} \left(\frac{\alpha_i}{h}\right)_x + \frac{\partial h}{\partial b} b_x.
\end{equation}
Here, the derivatives $\boldsymbol{v}_x$ can be computed directly from the known equilibrium variable representation.  The partial derivatives $\frac{\partial h}{\partial\boldsymbol{v}}$ are obtained using the implicit function theorem, since $h$ is implicitly defined by the quartic equation $q(h) = 0$ from \eqref{for:quartic}. For instance
\begin{equation}
    \frac{\partial h}{\partial E} = -\frac{\frac{\partial q}{\partial E}}{\frac{\partial q}{\partial h}}.
\end{equation}
Other partial derivatives are computed analogously.

\subsection{The fully-discrete well-balanced scheme}
For hyperbolic conservation laws, high-order accuracy in both space and time is often achieved by coupling a spatial discretization (e.g., DG or finite volume) with strong stability preserving Runge-Kutta (SSP-RK) or multistep time integration methods \cite{gottlieb2001, shuchiwang1988}. These time discretization strategies extend naturally to non-conservative hyperbolic systems as well.

In this work, we adopt a third-order SSP-RK method for the time discretization. The fully discrete formulation is given by
\begin{equation}\label{for:SSP-RK2}
    \begin{aligned}
     \int_{I_{j}}\boldsymbol{u}(\boldsymbol{v}^{(1)})\cdot\boldsymbol{\varphi}dx&=\int_{I_{j}}\boldsymbol{u}(\boldsymbol{v}^{n})\cdot\boldsymbol{\varphi}dx+\Delta t\mathrm{RHS}_{j}(\boldsymbol{v}^{n},\boldsymbol{\varphi}), \\
     \int_{I_{j}}\boldsymbol{u}(\boldsymbol{v}^{(2)})\cdot\boldsymbol{\varphi}dx&=\frac{3}{4}\int_{I_{j}}\boldsymbol{u}(\boldsymbol{v}^{n})\cdot\boldsymbol{\varphi}dx+\frac{1}{4}\left(\int_{I_{j}}\boldsymbol{u}(\boldsymbol{v}^{(1)})\cdot\boldsymbol{\varphi}dx+\Delta t\mathrm{RHS}_{j}(\boldsymbol{v}^{(1)},\boldsymbol{\varphi})\right), \\
     \int_{I_{j}}\boldsymbol{u}(\boldsymbol{v}^{n+1})\cdot\boldsymbol{\varphi}dx&=\frac{1}{3}\int_{I_{j}}\boldsymbol{u}(\boldsymbol{v}^{n})\cdot\boldsymbol{\varphi}dx+\frac{2}{3}\left(\int_{I_{j}}\boldsymbol{u}(\boldsymbol{v}^{(2)})\cdot\boldsymbol{\varphi}dx+\Delta t\mathrm{RHS}_{j}(\boldsymbol{v}^{(2)},\boldsymbol{\varphi})\right),
\end{aligned}
\end{equation}
where $\mathrm{RHS}_{j}(\boldsymbol{v}^{n},\boldsymbol{\varphi}) = \mathrm{RHS}_{j}^{M} \text{ or } \mathrm{RHS}_{j}^{S}$ refers to the spatial discretization operator defined in \eqref{for:RHS2} and \eqref{for:RHS1}, corresponding to the moving or still water equilibrium cases, respectively.
The $\mathrm{RHS}_{j}^{M}$ case also preserves the still water equilibrium but involves higher computational cost. Therefore, if the focus is only on the still water case, we use $\mathrm{RHS}_{j}^{S}$, which results in lower computational expense.

Due to the use of equilibrium variables $\boldsymbol{v}$, which are nonlinear functions of the conservative variables $\boldsymbol{u}$, each stage of the Runge-Kutta time stepping in \eqref{for:SSP-RK2} results in a nonlinear system of equations. As a result, Newton's method is employed locally in each element to solve for $\boldsymbol{v}^{(1)}$, $\boldsymbol{v}^{(2)}$, and $\boldsymbol{v}^{n}$ at each stage of the time integration.

\subsubsection{Newton's method}
The SSP-RK scheme can be expressed as a convex combination of multiple forward Euler steps. To illustrate the Newton iteration procedure, we focus on the forward Euler update, as represented by the first stage in \eqref{for:SSP-RK2}
\begin{equation}\label{for:1equation}
    \int_{I_j}\boldsymbol{u}(\boldsymbol{v}^{n+1})\cdot\boldsymbol{\varphi}dx=\mathcal{R}_j(\boldsymbol{v}^n,\boldsymbol{\varphi}),
\end{equation}
where
\begin{equation}\label{for:Lj}
    \mathcal{R}_j(\boldsymbol{v}^n,\boldsymbol{\varphi})=\int_{I_j}\boldsymbol{u}(\boldsymbol{v}^n)\cdot\boldsymbol{\varphi}dx+\Delta t\mathrm{RHS}_j^M(\boldsymbol{v}^n,\boldsymbol{\varphi}).
\end{equation}
This simplification allows us to demonstrate the nonlinear solve at each stage while retaining the essential structure of the method.

Let $\mathcal{R}_j^{[i]} = \mathcal{R}_j(\boldsymbol{v}^{n,[i]},\boldsymbol{\varphi}^{[i]}),i=1,\dots,N+3$ denote the $i$-th component of the right-hand side of the equation \eqref{for:1equation}, i.e. the equation satisfied by the $i$-th equilibrium variable.
Assuming $\boldsymbol{v}^n\in\mathcal{V}_h^k$ is given, the vector $\mathcal{R}_j$ can be explicitly calculated according to \eqref{for:Lj}. In particular, the quantity $hu_m^{n+1}$ can be obtained directly from $\mathcal{R}_j^{[2]}$. Furthermore, a straightforward calculation shows that $\mathcal{R}_j^{[N+3]}=0$, indicating that the bottom topography $b$ remains time-independent. Consequently, no additional numerical dissipation is introduced by the proposed PCDG scheme.

What remains to be determined are the components $E^{n+1},(\frac{\alpha_1}{h})^{n+1},\dots,(\frac{\alpha_N}{h})^{n+1}\in\mathcal{W}_h^k$, whose polynomial coefficients are computed iteratively by solving the following nonlinear system
\begin{equation}\label{for:nonlinSys}
\begin{aligned}
    & \int_{I_j}h^{n+1}\left(\sum_{i=1}^kE^{i,n+1}\varphi_1^i,hu_m^{n+1},\sum_{i=1}^k(\frac{\alpha_1}{h})^{i,n+1}\varphi_3^i,\dots,\sum_{i=1}^k(\frac{\alpha_N}{h})^{i,n+1}\varphi_{N+2}^i,b\right)\varphi_1dx = \mathcal{R}_j^{[1]}, \\
    & \int_{I_j}h\alpha_1^{n+1}\left(\sum_{i=1}^kE^{i,n+1}\varphi_1^i,hu_m^{n+1},\sum_{i=1}^k(\frac{\alpha_1}{h})^{i,n+1}\varphi_3^i,\dots,\sum_{i=1}^k(\frac{\alpha_N}{h})^{i,n+1}\varphi_{N+2}^i,b\right)\varphi_3dx = \mathcal{R}_j^{[3]}, \\
    & \hspace{7cm}\vdots \\
    & \int_{I_j}h\alpha_N^{n+1}\left(\sum_{i=1}^kE^{i,n+1}\varphi_1^i,hu_m^{n+1},\sum_{i=1}^k(\frac{\alpha_1}{h})^{i,n+1}\varphi_3^i,\dots,\sum_{i=1}^k(\frac{\alpha_N}{h})^{i,n+1}\varphi_{N+2}^i,b\right)\varphi_{N+2}dx = \mathcal{R}_j^{[N+2]}.
\end{aligned}
\end{equation}
For additional details on solving this system and selecting appropriate initial guesses to initiate the Newton iteration for \eqref{for:nonlinSys}, we refer the reader to the comprehensive discussion in \cite{zhang2023}.

\begin{remark}
   We provide a good initial guess for the Newton iteration. Given $\boldsymbol{v}^n\in\mathcal{V}_h^k$, we can construct
    \begin{equation}
        \widetilde{\boldsymbol{u}}(x)=(\widetilde{h},hu_m,\widetilde{h\alpha_1},\dots,\widetilde{h\alpha_N},b)\in\mathcal{V}_h^k,
    \end{equation}
   by solving the moment-matching condition
    \begin{equation}
        \int_{I_j}\widetilde{\boldsymbol{u}}(x)\cdot\boldsymbol{\varphi}dx=\mathcal{R}_j(\boldsymbol{v}^n,\boldsymbol{\varphi}).
    \end{equation}
    Using the transformation \eqref{for:equ2cons}, we compute the corresponding equilibrium variables
    \begin{equation}
        \widetilde{E} = \frac{1}{2}\frac{(hu_m)^2}{\widetilde{h}^2}+g(\widetilde{h}+b)+\frac{3}{2}\sum_{i=1}^{N}\frac{1}{2i+1}\frac{(\widetilde{h\alpha_i})^2}{\widetilde{h}^2},
    \end{equation}
    \begin{equation}
        \widetilde{\frac{\alpha_i}{h}} = \frac{\widetilde{h\alpha_i}}{\widetilde{h}^2},\quad i = 1,\dots,N.
    \end{equation}
    Finally, the $L^2$ projections of $\widetilde{E},\widetilde{\frac{\alpha_i}{h}}$ onto the polynomial space $\mathcal{W}_h^k$ are used as initial guesses for the Newton iteration.
\end{remark}

\subsubsection{Slope limiter}
In the discontinuous Galerkin (DG) framework, slope limiters are essential for suppressing spurious numerical oscillations, particularly in the presence of discontinuities. In this study, we adopt the total variation bounded (TVB) limiter \cite{cockburn1989, cockburn1998}, which is applied after each stage of the SSP-RK time discretization.

To preserve the generalized moving water equilibrium, we follow the strategy proposed in \cite{zhang2023} and modify the limiting process by applying it selectively to the local characteristic fields of the equilibrium variables
\begin{equation}
    \boldsymbol{v}^e = (E,hu_m,\frac{\alpha_1}{h},\dots,\frac{\alpha_N}{h})^T,\quad
    \boldsymbol{u}^e = (h,hu_m,h\alpha_1,\dots,h\alpha_N)^T.
\end{equation}
As in \cite{zhang2025}, the eigensystem used in the characteristic decomposition is obtained from
\begin{equation}
    \boldsymbol{A}_{j+1/2}=\left[\frac{\partial\boldsymbol{v}^e}{\partial\boldsymbol{u}^e}\mathcal{A}\left(\frac{\partial\boldsymbol{v}^e}{\partial\boldsymbol{u}^e}\right)^{-1}\right]_{\boldsymbol{u}^e_{j+1/2}},
\end{equation}
where $\mathcal{A}$ is the Jacobian matrix given in \eqref{for:jacobianmat}, and $\boldsymbol{u}^e_{j+1/2}$ denotes the average state, chosen as the simple mean of the left and right values. This tailored approach ensures that the limiter does not interfere with the delicate balance conditions required for equilibrium preservation.

The TVB limiter procedure consists of two main steps. First, we identify cells that require limiting by examining the equilibrium variables $\boldsymbol{v}^e$, represented as piecewise polynomials in the DG space at each time step. Second, the limiter is applied only to those cells flagged as troubled, and it operates directly on the polynomials $\boldsymbol{v}^e$.  If the numerical solution satisfies a steady state such that  $\boldsymbol{v} ^e$ is constant across the cell, the limiter is automatically deactivated. This design ensures that the limiting operation does not interfere with the equilibrium structure, thereby preserving the well-balanced property of the PCDG scheme.

\subsection{Well-balanced property}
We now introduce the following theorem to establish the well-balanced property of the proposed PCDG scheme for both moving water and still water equilibrium states.
\begin{theorem}\label{the:moving}
    The semi-discrete path-conservative discontinuous Galerkin method \eqref{for:semi-PCDG2}-\eqref{for:RHS2} with the modified numerical flux \eqref{for:LF-fluxmod} and path integral \eqref{for:gpath} is exact for the moving equilibrium state \eqref{for:moving}.
\end{theorem}

\begin{proof}
    Assume that the initial condition is an exact moving water equilibrium state, i.e.,
    \begin{equation}
        \widetilde{\boldsymbol{v}}(x)=\mathrm{constant},\quad
        \widetilde{\boldsymbol{\upsilon}}_{j+\frac{1}{2}}^-=\widetilde{\boldsymbol{\upsilon}}_{j+\frac{1}{2}}^+.
    \end{equation}
    Combining this with the linear segment path \eqref{for:linpath1}, we obtain
    \begin{equation}
        \int_{0}^{1}\mathcal{L}(\Phi_{j+\frac{1}{2}}(s))\frac{\partial}{\partial s}\widetilde{\boldsymbol{v}}(\Phi_{j+\frac{1}{2}}(s)) ds = 0.
    \end{equation}
    Thus, the path integral term \eqref{for:pathinter} reduces to the difference of fluxes
    \begin{equation}\label{for:gpathSimp}
    g_{\Phi,j+\frac12} = -\boldsymbol{f}(\boldsymbol{u}_{j+\frac{1}{2}}^{+})+\boldsymbol{f}(\boldsymbol{u}_{j+\frac{1}{2}}^{-}).
    \end{equation}
   By the definition in \eqref{for:umod}, we immediately have $\boldsymbol{u}_{j+\frac{1}{2}}^{*,+} = \boldsymbol{u}_{j+\frac{1}{2}}^{*,-}$, and the modified numerical flux from \eqref{for:LF-fluxmod} simplifies to
    \begin{equation}\label{for:LF-fluxmodSimp}
        \widehat{\boldsymbol{f}}_{j+\frac{1}{2}}^{\mathrm{mod}}=\frac{1}{2}\left(\boldsymbol{f}(\boldsymbol{u}_{j+\frac{1}{2}}^-)+\boldsymbol{f}(\boldsymbol{u}_{j+\frac{1}{2}}^+)\right).
    \end{equation}
    Substituting equations \eqref{for:gpathSimp} and \eqref{for:LF-fluxmodSimp} into the definition of $\mathrm{RHS}_{j}^{M}$ in \eqref{for:RHS2}, we derive
    \begin{align*}
    & \mathrm{RHS}_{j}^{M} \\
    &= \int_{\mathcal{I}_{j}}\boldsymbol{f}(\boldsymbol{u}(\boldsymbol{v}))\cdot\boldsymbol{\varphi}_{x}dx-\frac{1}{2}\left(\boldsymbol{f}(\boldsymbol{u}_{j+\frac{1}{2}}^-)+\boldsymbol{f}(\boldsymbol{u}_{j+\frac{1}{2}}^+)\right)\cdot\boldsymbol{\varphi}_{j+\frac{1}{2}}^-+\frac{1}{2}\left(\boldsymbol{f}(\boldsymbol{u}_{j-\frac{1}{2}}^-)+\boldsymbol{f}(\boldsymbol{u}_{j-\frac{1}{2}}^+)\right)\cdot\boldsymbol{\varphi}_{j-\frac{1}{2}}^+ \\
    & \quad -\int_{\mathcal{I}_{i}}\mathcal{G}(\boldsymbol{u}(\boldsymbol{v}))\boldsymbol{u}(\boldsymbol{v})_{x}\cdot\boldsymbol{\varphi}dx-\frac{1}{2}\boldsymbol{\varphi}_{j+\frac{1}{2}}^-\left(-\boldsymbol{f}(\boldsymbol{u}_{j+\frac{1}{2}}^+)+\boldsymbol{f}(\boldsymbol{u}_{j+\frac{1}{2}}^-)\right)-\frac{1}{2}\boldsymbol{\varphi}_{j-\frac{1}{2}}^+\left(-\boldsymbol{f}(\boldsymbol{u}_{j-\frac{1}{2}}^+)+\boldsymbol{f}(\boldsymbol{u}_{j-\frac{1}{2}}^-)\right) \\
    & = \int_{I_{j}}\boldsymbol{f}(\boldsymbol{u}(\boldsymbol{v}))\cdot\boldsymbol{\varphi}_{x}dx-\boldsymbol{f}(\boldsymbol{u}_{j+\frac{1}{2}}^{-})\cdot\boldsymbol{\varphi}_{j+\frac{1}{2}}^{-}+\boldsymbol{f}(\boldsymbol{u}_{j-\frac{1}{2}}^{+})\cdot\boldsymbol{\varphi}_{j-\frac{1}{2}}^{+}-\int_{\mathcal{I}_{j}}\mathcal{G}(\boldsymbol{u}(\boldsymbol{v}))\boldsymbol{u}(\boldsymbol{v})_{x}\cdot\boldsymbol{\varphi}dx \\
    & = -\int_{I_{j}}\partial_{x}\boldsymbol{f}(\boldsymbol{u}(\boldsymbol{v}))\cdot\boldsymbol{\varphi}dx-\int_{\mathcal{I}_{j}}\mathcal{G}(\boldsymbol{u}(\boldsymbol{v}))\boldsymbol{u}(\boldsymbol{v})_{x}\cdot\boldsymbol{\varphi}dx \\
    & = -\int_{I_{j}}\left(\partial_{x}\boldsymbol{f}(\boldsymbol{u}(\boldsymbol{v}))+\mathcal{G}(\boldsymbol{u}(\boldsymbol{v}))\boldsymbol{u}(\boldsymbol{v})_{x}\right)\cdot\boldsymbol{\varphi}dx=0.
    \end{align*}
    The final equality follows from the fact that the solution is in exact equilibrium. This completes the proof of the well-balanced property of the PCDG scheme for the moving water equilibrium state.
    % The still water case, as a special case of moving water, is thus also covered by Theorem \ref{the:moving}.

\end{proof}

\begin{theorem}
    The semi-discrete path-conservative discontinuous Galerkin method \eqref{for:semi-PCDG1}-\eqref{for:RHS1} is exact for still water equilibrium states \eqref{for:still}.
\end{theorem}

\begin{proof}
   Assume that the initial values are in an exact still water steady state, i.e.,
    \begin{equation}
        \boldsymbol{w}(x)=\mathrm{constant},\quad\boldsymbol{w}_{j+\frac{1}{2}}^-=\boldsymbol{w}_{j+\frac{1}{2}}^+.
    \end{equation}
   From the definitions of the numerical flux \eqref{for:LF-flux} and the path integral term \eqref{for:integralCal2}, it follows that
%    \begin{small}
    \begin{equation}
        \widehat{\boldsymbol{f}}_{j+\frac{1}{2}}=\boldsymbol{f}(\boldsymbol{w}_{j+\frac{1}{2}}^-)=\boldsymbol{f}(\boldsymbol{w}_{j+\frac{1}{2}}^+),\quad
        g_{\Phi,j+\frac{1}{2}}=0.
    \end{equation}
   Hence, the $\mathrm{RHS}_{j}^S$ in \eqref{for:RHS1} becomes
    \begin{equation}
        \begin{aligned}
        \mathrm{RHS}_{j}^S & =\int_{\mathcal{I}_{j}}\boldsymbol{f}(\boldsymbol{w})\cdot\boldsymbol{\varphi}_{x}dx-\boldsymbol{f}(\boldsymbol{w}_{j+\frac{1}{2}}^-)\cdot\boldsymbol{\varphi}_{j+\frac{1}{2}}^-+\boldsymbol{f}(\boldsymbol{w}_{j-\frac{1}{2}}^+)\cdot\boldsymbol{\varphi}_{j-\frac{1}{2}}^+-\int_{\mathcal{I}_{j}}\mathcal{G}(\boldsymbol{w})\boldsymbol{w}_{x}\cdot\boldsymbol{\varphi}dx \\
         & =-\int_{\mathcal{I}_{j}}\boldsymbol{f}(\boldsymbol{w})_{x}\cdot\boldsymbol{\varphi}dx-\int_{\mathcal{I}_{j}}\mathcal{G}(\boldsymbol{w})\boldsymbol{w}_{x}\cdot\boldsymbol{\varphi}dx \\
         & =-\int_{\mathcal{I}_{i}}\left(\boldsymbol{f}(\boldsymbol{w})_{x}+\mathcal{G}(\boldsymbol{w})\boldsymbol{w}_{x}\right)\cdot\boldsymbol{\varphi}dx=0.
        \end{aligned}
    \end{equation}
%    \end{small}
The second equality results from integration by parts. The final equality holds because $\boldsymbol{w}$ is a steady state solution. This completes the proof of the well-balanced property of the scheme for the still water equilibrium.
\end{proof}
Since the still water case is technically a special case of the moving water equilibrium, it is also preserved in our scheme.
The fully discrete PCDG scheme also preserves the exact equilibrium states. This well-balanced property can be established by induction, and the proof is omitted here for brevity.
\begin{theorem}
    The fully discrete path-conservative discontinuous Galerkin method, consisting of the third-order SSP-RK time discretization \eqref{for:SSP-RK2} and either the spatial discretization \eqref{for:semi-PCDG1} or \eqref{for:semi-PCDG2}, exactly preserves the still and moving water equilibrium states, i.e., \eqref{for:still} and \eqref{for:moving}.
\end{theorem}

% \begin{theorem}
%     The fully discrete path-conservative discontinuous Galerkin method, consisting of the spatial discretization \eqref{for:semi-PCDG2} and the third-order SSP-RK time discretization \eqref{for:SSP-RK2}, exactly preserves the moving water equilibrium state \eqref{for:moving}.
% \end{theorem}

% \begin{theorem}
%     The fully discrete path-conservative discontinuous Galerkin method, consisting of the spatial discretization \eqref{for:semi-PCDG1} and the third-order SSP-RK time discretization \eqref{for:SSP-RK2}, exactly preserves the still water equilibrium state \eqref{for:still}.
% \end{theorem}

\section{Numerical examples}
\label{se-nu}
In this section, we numerically verify the still water and moving water equilibrium preserving properties of the proposed path-conservative discontinuous Galerkin methods. The schemes are referred to as PCDG-still and PCDG-moving, respectively. All tests are conducted using $N=2$ moments, except for the last test case, which uses $N=8$ to demonstrate the capabilities of the code for large number of equations.

Unless otherwise specified, we employ piecewise quadratic ($P^2$) polynomial basis functions for spatial discretization. For all simulations, the CFL number is fixed at 0.05. The TVD slope limiter is used, except in the accuracy tests where the solution is smooth and no limiter is applied. The gravitational constant is set to $g =9.812 m/s^2$, unless stated otherwise. Free boundary conditions are imposed in all subsequent examples.
Notably, the tolerance in the Newton solver is set to $\mathrm{TOL} = 10^{-16}$ for Examples~\ref{exmp:still} and~\ref{exmp:moving}, to resolve the equilibrium states with high precision. For all other examples, a tolerance of $\mathrm{TOL} = 10^{-10}$ is used.

\begin{example}\label{exmp:accuracy}
 Accuracy test.
\end{example}
To test the accuracy of our proposed well-balanced PCDG schemes for a smooth solution, we choose the same bottom topography and initial conditions as in \cite{xing2014536}, with additional $N = 2$ non-zero moments
\begin{equation}\label{for:Ex1init}
\begin{aligned}
    & b(x)=\sin^2(\pi x),\quad h(x,0)=5+e^{\cos(2\pi x)},\quad hu_m(x,0)=\sin(\cos(2\pi x)), \\
    & \frac{\alpha_1}h(x,0)=0.25,\quad \frac{\alpha_2}h(x,0)=0.25.
\end{aligned}
\end{equation}

We perform simulations on the domain $[0, 1]$ with periodic boundary conditions and compute the solution up to $t = 0.01$ to avoid the formation of shocks. The numerical solution obtained by the third-order PCDG-still method using $nx=12,800$ grid cells is taken as the reference solution. We compute the  $L^1$ errors and corresponding convergence orders for the water height $h$, discharge $m=hu_m$ and the moment coefficients $\alpha_1$, $\alpha_2$.
% The results for both PCDG-still and PCDG-moving schemes are summarized in Tables~\ref{tab:accT1}-\ref{tab:accT2}, demonstrating the high-order accuracy and optimal convergence rates of the proposed methods.
The results for both the PCDG-still and PCDG-moving schemes are summarized in Tables~\ref{tab:accT1}-\ref{tab:accT2}, demonstrating third-order convergence consistent with the theoretical expectation for $P^2$ polynomial basis functions. This confirms that the proposed methods achieve the optimal convergence rate.

\begin{table}[!htb]
    \centering
    \caption{\label{tab:accT1}Example \ref{exmp:accuracy}. $L^1$ errors and numerical orders of accuracy with initial conditions \eqref{for:Ex1init}, using PCDG-still method.}
    \begin{tabular}{llllllllllll}
    \toprule
    \multirow{2}{*}{$nx$} & $h$           &       &  & $m$           &       &  & $\alpha_1$  &       &  & $\alpha_2$  &       \\ \cline{2-3} \cline{5-6} \cline{8-9} \cline{11-12}
                    & $L^1$ error & order &  & $L^1$ error & order &  & $L^1$ error & order &  & $L^1$ error & order \\
    \midrule
   20 & 1.3807e-04 &     / & & 6.3944e-04 &    / & & 4.7000e-04 &    / & & 4.7000e-04 &    / \\
   40 & 1.9044e-05 &  2.86 & & 7.5541e-05 & 3.08 & & 6.4906e-05 & 2.86 & & 6.4906e-05 & 2.86 \\
   80 & 2.4558e-06 &  2.96 & & 9.2914e-06 & 3.02 & & 8.3444e-06 & 2.96 & & 8.3444e-06 & 2.96 \\
   160 & 3.1551e-07 & 2.96 & & 1.1499e-06 & 3.01 & & 1.0539e-06 & 2.99 & & 1.0539e-06 & 2.99 \\
   320 & 4.1042e-08 & 2.94 & & 1.4268e-07 & 3.01 & & 1.3234e-07 & 2.99 & & 1.3234e-07 & 2.99 \\
   640 & 5.5751e-09 & 2.88 & & 1.7719e-08 & 3.01 & & 1.6897e-08 & 2.97 & & 1.6897e-08 & 2.97 \\
    % \midrule
    % \multirow{2}{*}{$P^3$} & h           &       &  & m           &       &  & $\alpha_1$  &       &  & $\alpha_2$  &       \\ \cline{2-3} \cline{5-6} \cline{8-9} \cline{11-12}
    %                 & $L^1$ error & order &  & $L^1$ error & order &  & $L^1$ error & order &  & $L^1$ error & order \\
    % \midrule
    % 20 & 6.8477e-06 &     / & &  2.5392e-05 &     / & &  2.6731e-05 &     / & &  2.6731e-05 &     / \\
    % 40 & 4.1606e-07 &  4.04 & &  1.5665e-06 &  4.02 & &  1.6330e-06 &  4.03 & &  1.6330e-06 &  4.03 \\
    % 80 & 2.5257e-08 &  4.04 & &  9.3384e-08 &  4.07 & &  9.8401e-08 &  4.05  & & 9.8401e-08 &  4.05 \\
    % 160 & 1.5772e-09 &  4.00 & &  5.7985e-09 &  4.01 & &  6.1551e-09 &  4.00 & &  6.1551e-09 &  4.00 \\
    % 320 & 9.8997e-11 &  3.99 & &  3.6064e-10 &  4.01 & &  3.8499e-10 &  4.00 & &  3.8499e-10 &  4.00 \\
    % 640 & 7.1489e-12 &  3.79 & &  2.2798e-11 &  3.98 & &  2.6688e-11 &  3.85 & &  2.6688e-11 &  3.85 \\
    \bottomrule
    \end{tabular}
\end{table}

\begin{table}[!htb]
    \centering
    \caption{\label{tab:accT2}Example \ref{exmp:accuracy}. $L^1$ errors and numerical orders of accuracy with initial conditions \eqref{for:Ex1init}, using PCDG-moving method.}
    \begin{tabular}{llllllllllll}
    \toprule
    \multirow{2}{*}{$nx$} & $h$           &       &  & $m $          &       &  & $\alpha_1$  &       &  & $\alpha_2$  &       \\ \cline{2-3} \cline{5-6} \cline{8-9} \cline{11-12}
                    & $L^1$ error & order &  & $L^1$ error & order &  & $L^1$ error & order &  & $L^1$ error & order \\
    \midrule
   20  &  1.3863e-04 &       / & &  1.8362e-04 &       / & &  4.7173e-04 &      /  & & 4.7173e-04 &       /  \\
   40  &  1.7567e-05 &  2.98 & &  4.0350e-05 &  2.19 & &  5.9376e-05 &  2.99 & &  5.9376e-05 &  2.99 \\
   80  &  2.2729e-06 &  2.95 & &  9.0823e-06 &  2.15 & &  7.6330e-06 &  2.96 & &  7.6330e-06 &  2.96 \\
   160 &  2.7076e-07 &  3.07 & &  1.0963e-06 &  3.05 & &  9.1128e-07 &  3.07 & &  9.1128e-07 &  3.07 \\
   320 &  3.3876e-08 &  3.00 & &  1.3054e-07 &  3.07 & &  1.1412e-07 &  3.00 & &  1.1412e-07 &  3.00 \\
   640 &  4.2378e-09 &  3.00 & &  1.6203e-08 &  3.01 & &  1.4277e-08 &  3.00 & &  1.4277e-08 &  3.00 \\
    % \midrule
    % \multirow{2}{*}{$P^3$} & h           &       &  & m           &       &  & $\alpha_1$  &       &  & $\alpha_2$  &       \\ \cline{2-3} \cline{5-6} \cline{8-9} \cline{11-12}
    %                 & $L^1$ error & order &  & $L^1$ error & order &  & $L^1$ error & order &  & $L^1$ error & order \\
    % \midrule
    % 20 & 6.8477e-06 &     / & &  2.5392e-05 &     / & &  2.6731e-05 &     / & &  2.6731e-05 &     / \\
    % 40 & 4.1606e-07 &  4.04 & &  1.5665e-06 &  4.02 & &  1.6330e-06 &  4.03 & &  1.6330e-06 &  4.03 \\
    % 80 & 2.5257e-08 &  4.04 & &  9.3384e-08 &  4.07 & &  9.8401e-08 &  4.05  & & 9.8401e-08 &  4.05 \\
    % 160 & 1.5772e-09 &  4.00 & &  5.7985e-09 &  4.01 & &  6.1551e-09 &  4.00 & &  6.1551e-09 &  4.00 \\
    % 320 & 9.8997e-11 &  3.99 & &  3.6064e-10 &  4.01 & &  3.8499e-10 &  4.00 & &  3.8499e-10 &  4.00 \\
    % 640 & 7.1489e-12 &  3.79 & &  2.2798e-11 &  3.98 & &  2.6688e-11 &  3.85 & &  2.6688e-11 &  3.85 \\
    \bottomrule
    \end{tabular}
\end{table}

%\subsection*{Example 5.2. Still water equilibrium test}
\begin{example}\label{exmp:still}
Still water equilibrium test.
\end{example}

We use the following example to verify the well-balanced property of the proposed method for the still water equilibrium. The test is conducted on the spatial domain $[0,25]$, considering two different bottom topographies. The first case uses a continuous bottom topography defined by

\begin{equation}\label{for:bottomC}
    \left.b_0(x)=\left\{
    \begin{array}{ll}
    0.2-0.05(x-10)^2, & \text{if}\quad x\in[8,12], \\
    0, & \text{otherwise},
    \end{array}\right.\right.
\end{equation}
and the second case uses a discontinuous bottom topography
\begin{equation}\label{for:bottomD}
    b(x)=\left\{
    \begin{array}{ll}
    0.2,& \text{if}\quad x\in[8,12], \\
    0, & \text{otherwise}.
\end{array}\right.
\end{equation}
The initial conditions are given at equilibrium by
\begin{equation}
    h(x,0) = 2-b(x),\quad hu_m(x,0)=0,\quad\alpha_1(x,0) = 0,\quad\alpha_2(x,0) = 0.
\end{equation}

We use $nx=100$ uniform grid cells to compute the numerical solution at $t = 1$, employing a quadratic polynomial basis. The $L^1$ and $L^\infty$ errors of the water surface elevation $h+b$ and velocity $u$ are calculated. The results for the two proposed schemes, PCDG-still and PCDG-moving, over both continuous and discontinuous bottom topographies are reported in Tables~\ref{tab:stillerr1} and \ref{tab:stillerr2}, respectively. The results confirm that both schemes exactly preserve the still water equilibrium. The PCDG-moving method occasionally shows slightly larger numerical errors due to the nonlinear Newton iteration involved in solving for the equilibrium variables.

\begin{table}[!htbp]
\centering
\caption{\label{tab:stillerr1}Example \ref{exmp:still}. $L^1$ and $L^\infty$ errors with continuous bottom topography \eqref{for:bottomC}.}
\begin{tabular}{lllll}
\toprule
Scheme ($nx=100$) & $||\Delta (h+b)||_{L^1}$ & $||\Delta u||_{L^1}$ & $||\Delta (h+b)||_{L^\infty}$ & $||\Delta u||_{L^\infty}$ \\
\midrule
PCDG-still       & 1.8544e-15 & 3.3776e-15 & 1.1102e-15 & 1.4690e-15 \\
PCDG-moving      & 4.8248e-14 & 4.8229e-14 & 4.8849e-15 & 4.8849e-15 \\
\bottomrule
\end{tabular}
\end{table}

\begin{table}[!htbp]
\centering
\caption{\label{tab:stillerr2}Example \ref{exmp:still}. $L^1$ and $L^\infty$ errors with discontinuous bottom topography \eqref{for:bottomD}.}
\begin{tabular}{lllll}
\toprule
Scheme ($nx=100$) & $||\Delta (h+b)||_{L^1}$ & $||\Delta u||_{L^1}$ & $||\Delta (h+b)||_{L^\infty}$ & $||\Delta u||_{L^\infty}$ \\
\midrule
PCDG-still       & 5.9646e-15 & 3.8598e-14 & 2.6645e-15 & 8.2258e-15 \\
PCDG-moving      & 3.8805e-14 & 3.8805e-14 & 4.8850e-15 & 4.8850e-15 \\
\bottomrule
\end{tabular}
\end{table}

%\subsection*{Example 5.3. Small perturbation of still water equilibrium}
\begin{example}\label{exmp:still_pert}
Small perturbation of the still water equilibrium.
\end{example}

The following example is used to demonstrate the capability of the proposed schemes in accurately capturing the propagation of small perturbations around the still water equilibrium steady state. We adopt the same test case as in \cite{xing2006new}. The computational domain is $[0,2]$, and the bottom topography and initial conditions are given by
\begin{equation}
    b(x)=
    \left\{
    \begin{array}{ll}
    0.25(\cos(10\pi(x-1.5))+1),& \text{if~}1.4\leq x\leq1.6, \\
    0,& \text{otherwise},
    \end{array}\right.
\end{equation}
\begin{equation}
    hu_m(x,0)=0, \quad\alpha_1(x,0)=\alpha_2(x,0)=0, \quad h(x,0)=\left\{
    \begin{array}{ll}
    1-b(x)+\epsilon,& \text{if~}1.1\leq x\leq1.2, \\
    1-b(x),& \text{otherwise.}
    \end{array}\right.
\end{equation}
where $\epsilon$ denotes the perturbation amplitude.

We consider two cases with $\epsilon=0.2$ (large pulse) and $\epsilon=0.001$ (small pulse), and evaluate the solution at $t=0.2$, using transmissive boundary conditions. The simulations are conducted on a mesh with $nx=200$ uniform cells, and a finer grid with $nx=3000$ cells is used to generate the reference solution. Figures~\ref{fig:perturStillBig} and \ref{fig:perturStillSmall} show the surface elevation $h+b$ (left) and the discharge $hu_m$ (right) for both the large and small pulse cases, respectively. As expected, the proposed well-balanced schemes accurately capture small perturbations of the still water equilibrium without generating spurious oscillations near discontinuities, confirming their robustness and fidelity.

\begin{figure}[htb]
\centering
\includegraphics[width=1\linewidth]{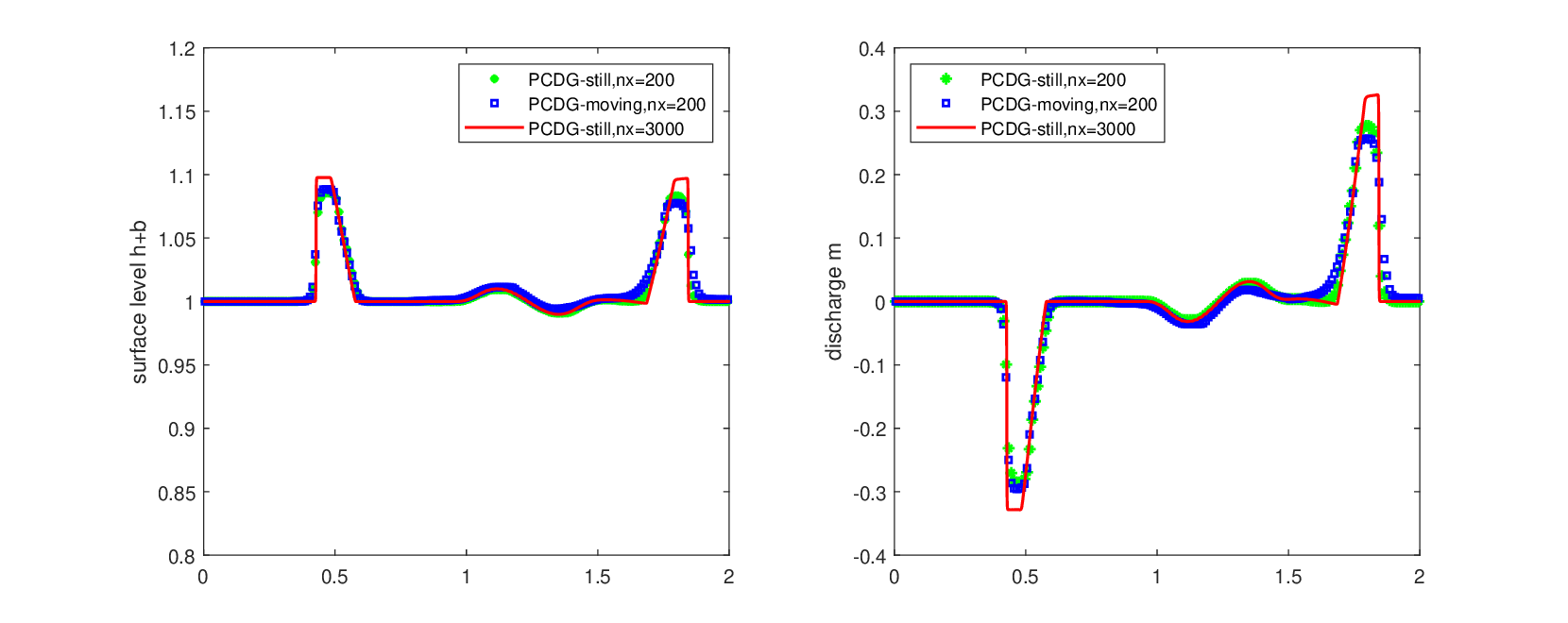}
\caption{\label{fig:perturStillBig} Example \ref{exmp:still_pert}. Small perturbation of still water equilibrium for the big pulse $\epsilon=0.2$, using $nx=200$ and $nx=3000$ grid cells at time $t=0.2$.}
\end{figure}

\begin{figure}[htb]
\centering
\includegraphics[width=1\linewidth]{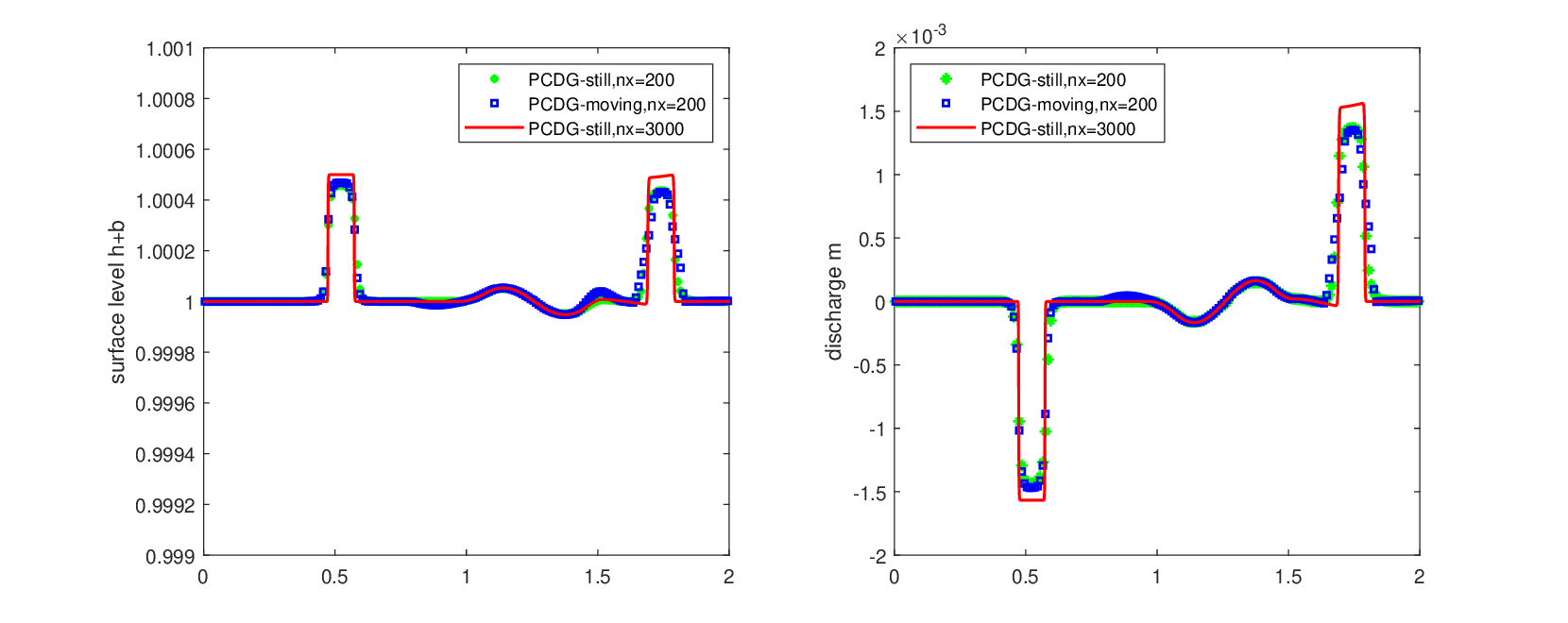}
\caption{\label{fig:perturStillSmall}Example \ref{exmp:still_pert}. Small perturbation of still water equilibrium for the small pulse $\epsilon=0.001$, using $nx=200$ and $nx=3000$ grid cells at time $t=0.2$.}
\end{figure}

%\subsection*{Example 5.4. Moving water equilibrium test}
\begin{example}\label{exmp:moving}
Moving water equilibrium test.
\end{example}

The following examples are designed to demonstrate the well-balanced property of the proposed scheme for the moving water equilibrium. Similar to the setup in Example~\ref{exmp:still}, we consider the same continuous and discontinuous bottom topographies defined in \eqref{for:bottomC} and \eqref{for:bottomD}, respectively.

We examine three representative steady-state flow regimes: subcritical, supercritical, and transcritical. These classical test cases, as presented in \cite{xing2014536, Cheng2016}, are commonly used to evaluate the performance of numerical methods in preserving moving water equilibria.
The initial conditions for all three cases are detailed below. To specifically test the well-balanced property for the shallow water linearized moment equations, we adapt the original setups by introducing $N = 2$ non-zero moments in the first two examples. These adapted cases are used for the first time to validate the well-balanced property of a discontinuous Galerkin scheme for the SWLME with non-zero moments.

\subsubsection*{(a) Subcritical flow}
The initial conditions are given by
\begin{equation}
    E_0 = 22.09805,\quad (hu_m)_0 = 4.42,\quad (\frac{\alpha_1}{h})_0 = 0.1, \quad (\frac{\alpha_2}{h})_0 = -0.1.
\end{equation}

\subsubsection*{(b) Supercritical flow}
The initial conditions are given by
\begin{equation}
    E_0 = 91.6320,\quad (hu_m)_0 = 24,\quad (\frac{\alpha_1}{h})_0 = 0.1, \quad (\frac{\alpha_2}{h})_0 = -0.1.
\end{equation}

\subsubsection*{(c) Transcritical flow}
The initial conditions are given by
\begin{equation}
    E_0 = 11.0907140397782,\quad (hu_m)_0 = 1.53,\quad (\frac{\alpha_1}{h})_0 = (\frac{\alpha_2}{h})_0 = 0.
\end{equation}
To find the correct roots of this quartic equation, we give the initial type of the flow
\begin{equation}
    \left\{
    \begin{array}{ll}
    \text{subcritical flow},& \text{if}\quad x<10, \\
    \text{supercritical flow},& \text{if}\quad x>10,
    \end{array}\right.
\end{equation}
for the continuous bottom topography \eqref{for:bottomC}, and
\begin{equation}
    \left\{
    \begin{array}{ll}
    \text{subcritical flow},& \text{if}\quad x<8, \\
    \text{sonic flow},& \text{if}\quad 8\leq x\leq12, \\
    \text{supercritical flow},& \text{if}\quad x>12,
    \end{array}\right.
\end{equation}
for the discontinuous bottom topography \eqref{for:bottomD}. In this way, the initial value can be determined analytically according to the flow type.

We compute the numerical solution at $t = 1$ using $nx=100$ uniform grid cells and evaluate the $L^1$ and $L^\infty$ errors of the equilibrium variables: the energy $E$, the discharge $hu_m$ and the moment coefficients $\frac{\alpha_1}{h}$ and $\frac{\alpha_2}{h}$. The computational results of three flow types over two different bottom topographies are presented in Tables \eqref{tab:movingerr1c}-\eqref{tab:movingerr3d}. The results show that the numerical errors produced by the proposed PCDG-moving scheme are on the order of machine precision in all cases, demonstrating excellent preservation of the moving water equilibrium regardless of the bottom topography. In contrast, the PCDG-still scheme consistently produces larger errors, as expected, since it is not designed to preserve moving water equilibria. For the transcritical case, note that the moment variables $\frac{\alpha_1}{h}$ and $\frac{\alpha_2}{h}$ are not present in the initial data. As a result, their errors remain exactly zero throughout the simulation.

\begin{table}[!htb]
\centering
\caption{\label{tab:movingerr1c}Example \ref{exmp:moving}. Subcritical flow: $L^1$ and $L^\infty$ errors with continuous bottom topography.}
\begin{tabular}{lllll}
\toprule
PCDG-moving & $E$ & $hu_m$ & $\frac{\alpha_1}{h}$ & $\frac{\alpha_2}{h}$ \\
\midrule
$L_1$      & 3.0291e-12 & 4.7727e-13 & 1.6772e-13 & 3.5504e-13 \\
$L_\infty$ & 3.3751e-13 & 5.3735e-14 & 2.7756e-14 & 3.8858e-14 \\
\midrule
PCDG-still & $E$ & $hu_m$ & $\frac{\alpha_1}{h}$ & $\frac{\alpha_2}{h}$ \\
\midrule
$L_1$      & 1.3705e-08 & 4.2868e-08 & 2.0582e-08 & 2.0582e-08 \\
$L_\infty$ & 1.3287e-08 & 3.6694e-08 & 2.0283e-08 & 2.0283e-08 \\
\bottomrule
\end{tabular}
\end{table}

\begin{table}[!htb]
\centering
\caption{\label{tab:movingerr1d}Example \ref{exmp:moving}. Subcritical flow: $L^1$ and $L^\infty$ errors with discontinuous bottom topography.}
\begin{tabular}{lllll}
\toprule
PCDG-moving & $E$ & $hu_m$ & $\frac{\alpha_1}{h}$ & $\frac{\alpha_2}{h}$ \\
\midrule
 $L_1$      & 8.9823e-12 & 4.0705e-13 & 1.1185e-13 & 4.5471e-13 \\
 $L_\infty$ & 4.4764e-13 & 6.4837e-14 & 1.2657e-14 & 3.4750e-14 \\
\midrule
PCDG-still & $E$ & $hu_m$ & $\frac{\alpha_1}{h}$ & $\frac{\alpha_2}{h}$ \\
\midrule
 $L_1$      & 9.4936e-05 & 3.5906e-05 & 1.3908e-06 & 1.3908e-06 \\
 $L_\infty$ & 1.0644e-04 & 2.5534e-05 & 1.4802e-06 & 1.4802e-06 \\
\bottomrule
\end{tabular}
\end{table}

\begin{table}[!htb]
\centering
\caption{\label{tab:movingerr2c}Example \ref{exmp:moving}. Supercritical flow: $L^1$ and $L^\infty$ errors with continuous bottom topography.}
\begin{tabular}{lllll}
\toprule
PCDG-moving & $E$ & $hu_m$ & $\frac{\alpha_1}{h}$ & $\frac{\alpha_2}{h}$ \\
\midrule
  $L_1$       &  3.5499e-11 & 5.1394e-12 & 3.4511e-13 & 2.5034e-13  \\
  $L_\infty$  &  2.7569e-12 & 9.2015e-13 & 3.2863e-14 & 2.7978e-14  \\
\midrule
PCDG-still & $E$ & $hu_m$ & $\frac{\alpha_1}{h}$ & $\frac{\alpha_2}{h}$ \\
\midrule
  $L_1$       &  1.7394e-07 & 1.9073e-08 & 6.4261e-10 & 6.4261e-10 \\
  $L_\infty$  &  1.8614e-07 & 1.6443e-08 & 6.9139e-10 & 6.9139e-10 \\
\bottomrule
\end{tabular}
\end{table}

\begin{table}[!htb]
\centering
\caption{\label{tab:movingerr2d}Example \ref{exmp:moving}. Supercritical flow: $L^1$ and $L^\infty$ errors with discontinuous bottom topography.}
\begin{tabular}{lllll}
\toprule
PCDG-moving & $E$ & $hu_m$ & $\frac{\alpha_1}{h}$ & $\frac{\alpha_2}{h}$ \\
\midrule
  $L_1$      &  7.1135e-11 & 3.4192e-12 & 3.7787e-13 & 3.5608e-13 \\
  $L_\infty$ &  3.9080e-12 & 5.0804e-13 & 2.6201e-14 & 3.7859e-14 \\
\midrule
PCDG-still & $E$ & $hu_m$ & $\frac{\alpha_1}{h}$ & $\frac{\alpha_2}{h}$ \\
\midrule
  $L_1$      &  2.9191e-05 & 1.3231e-05 & 1.1083e-07 & 1.1083e-07 \\
  $L_\infty$ &  2.9417e-05 & 8.1318e-06 & 1.1486e-07 & 1.1486e-07 \\
\bottomrule
\end{tabular}
\end{table}

\begin{table}[!htb]
\centering
\caption{\label{tab:movingerr3c}Example \ref{exmp:moving}. Transcritical flow: $L^1$ and $L^\infty$ errors with continuous bottom topography.}
\begin{tabular}{lllll}
\toprule
PCDG-moving & $E$ & $hu_m$ & $\frac{\alpha_1}{h}$ & $\frac{\alpha_2}{h}$ \\
\midrule
$L_1$      & 1.8370e-12 & 4.9100e-13 & 0 & 0 \\
$L_\infty$ & 6.2794e-12 & 1.5514e-12 & 0 & 0 \\
\midrule
PCDG-still & $E$ & $hu_m$ & $\frac{\alpha_1}{h}$ & $\frac{\alpha_2}{h}$ \\
\midrule
 $L_1$      & 6.3949e-07	& 5.3418e-08 & 0 & 0 \\
 $L_\infty$ & 6.9161e-07	& 4.4539e-08 & 0 & 0 \\
\bottomrule
\end{tabular}
\end{table}

\begin{table}[!htb]
\centering
\caption{\label{tab:movingerr3d}Example \ref{exmp:moving}. Transcritical flow: $L^1$ and $L^\infty$ errors with discontinuous bottom topography.}
\begin{tabular}{lllll}
\toprule
PCDG-moving & $E$ & $hu_m$ & $\frac{\alpha_1}{h}$ & $\frac{\alpha_2}{h}$ \\
\midrule
$L_1$      & 6.5725e-14 & 3.7843e-14 & 0 & 0 \\
$L_\infty$ & 3.5527e-15 & 2.8355e-15 & 0 & 0 \\
\midrule
PCDG-still & $E$ & $hu_m$ & $\frac{\alpha_1}{h}$ & $\frac{\alpha_2}{h}$ \\
\midrule
 $L_1$      & 1.0799e-04 & 5.4681e-05 & 0 & 0    \\
 $L_\infty$ & 1.5798e-04 & 4.0166e-05 & 0 & 0 \\
\bottomrule
\end{tabular}
\end{table}

%\subsection*{Example 5.5. Small perturbation of the moving water equilibrium}
\begin{example}\label{exmp:moving_pert}
Small perturbation of the moving water equilibrium.
\end{example}
We consider a similar test case from \cite{noelle200729} to evaluate the ability of the proposed PCDG-moving scheme to capture the propagation of small perturbations around a moving water equilibrium--a property that cannot be preserved by schemes lacking well-balanced treatment for moving water states. Using the same initial condition as in Example~\ref{exmp:moving}, we examine two scenarios: one with the continuous bottom topography and one with the discontinuous bottom topography, as defined in \eqref{for:bottomC} and \eqref{for:bottomD}, respectively. In the third transcritical flow case, all moment coefficients are zero, reducing the system to the standard shallow water equations (SWE). Since this case has been studied in detail in \cite{zhang2023}, we omit the results here for brevity.
A small perturbation of magnitude $0.01$ is applied to the water height within the interval $[5.75, 6.25]$. Numerical simulations are performed up to $t = 1.5$ for the subcritical flow case and up to $t = 1.0$ for the supercritical flow case. The solution computed with the PCDG-moving scheme and $nx=1000$ uniform grid cells is used as the reference solution.

Figures~\ref{fig:SubCompareC}-\ref{fig:SupCompareD} display the numerical results obtained with $nx=200$ uniform cells for the surface elevation  $h+b$, energy $E$, and moment coefficients $\alpha_1$ and $\alpha_2$, along with the corresponding reference solutions. The results clearly demonstrate that the proposed PCDG-moving scheme accurately captures the evolution of small perturbations in the moving water equilibrium, including their effects on energy and moment dynamics. Moreover, the slight oscillations observed in the moment coefficients, introduced by the perturbation, are also well-resolved and can be further refined by mesh refinement.

\begin{figure}
    \centering
    \includegraphics[width=1\linewidth]{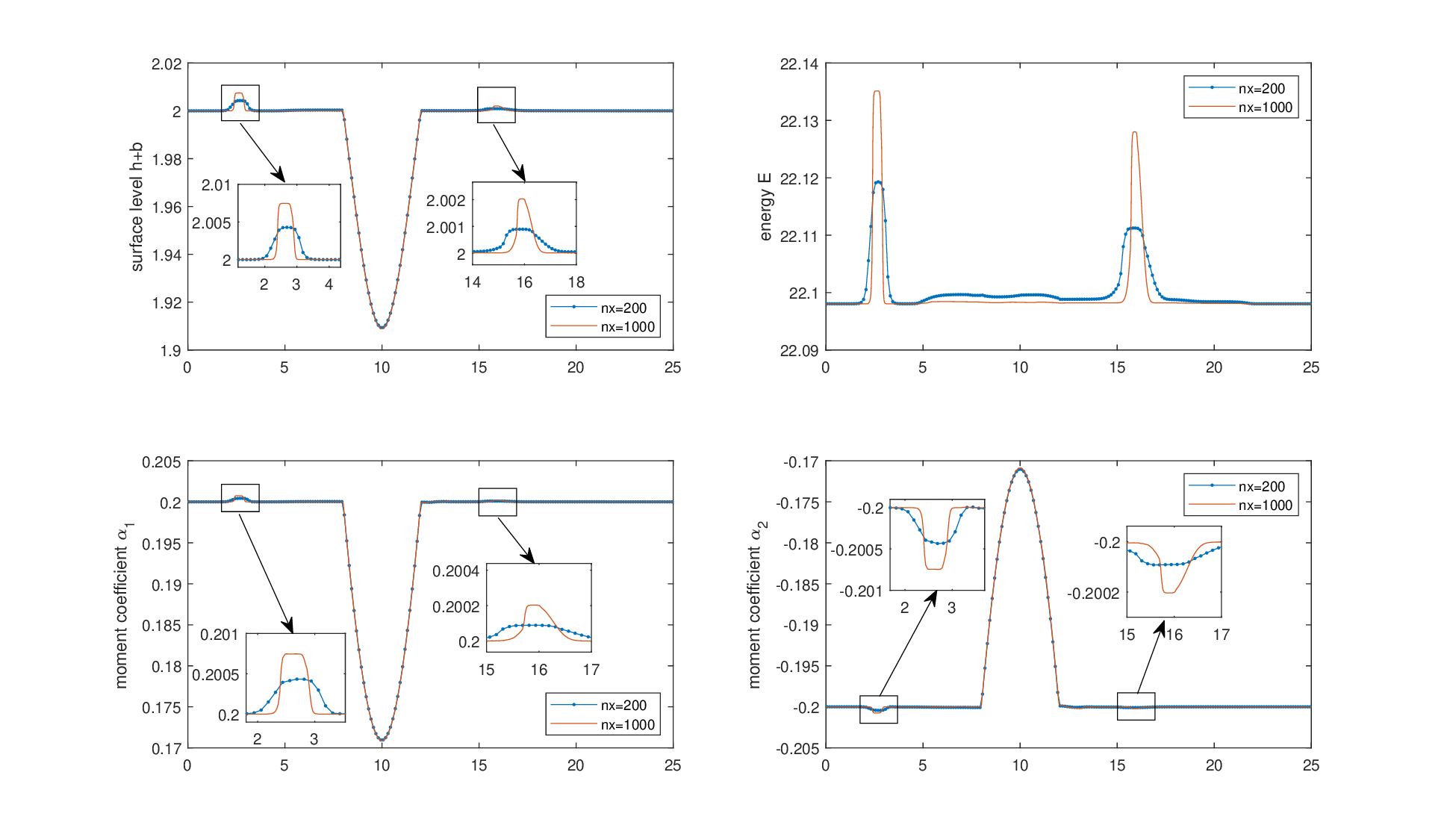}
    \caption{Example \ref{exmp:moving_pert}. Small perturbation of moving water equilibrium for subcritical flow with continuous bottom, using $nx=200$ and $nx=1000$ grid cells.}
    \label{fig:SubCompareC}
\end{figure}

\begin{figure}
    \centering
    \includegraphics[width=1\linewidth]{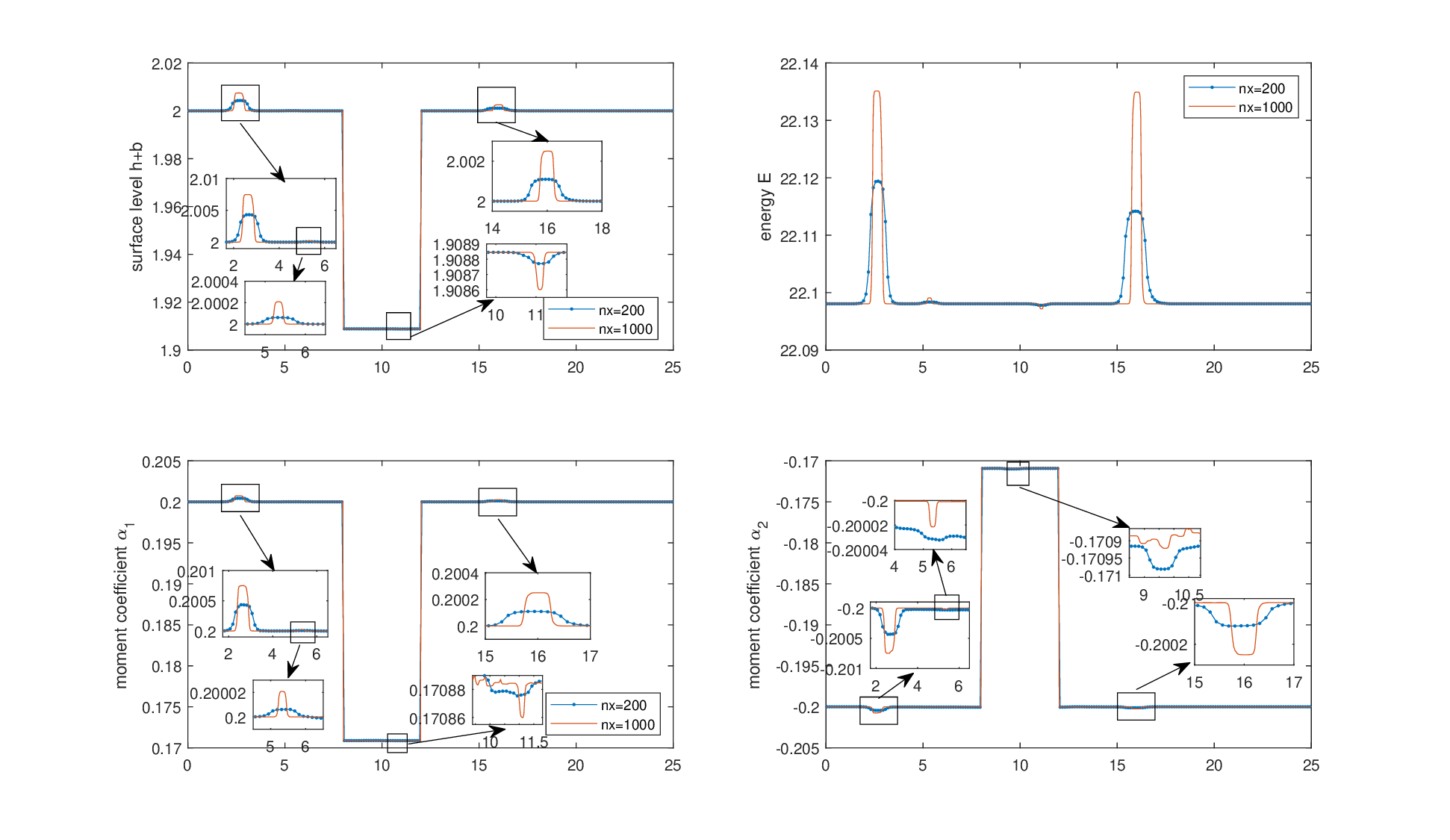}
    \caption{Example \ref{exmp:moving_pert}. Small perturbation of moving water equilibrium for subcritical flow with discontinuous bottom, using $nx=200$ and $nx=1000$ grid cells.}
    \label{fig:SubCompareD}
\end{figure}

\begin{figure}
    \centering
    \includegraphics[width=1\linewidth]{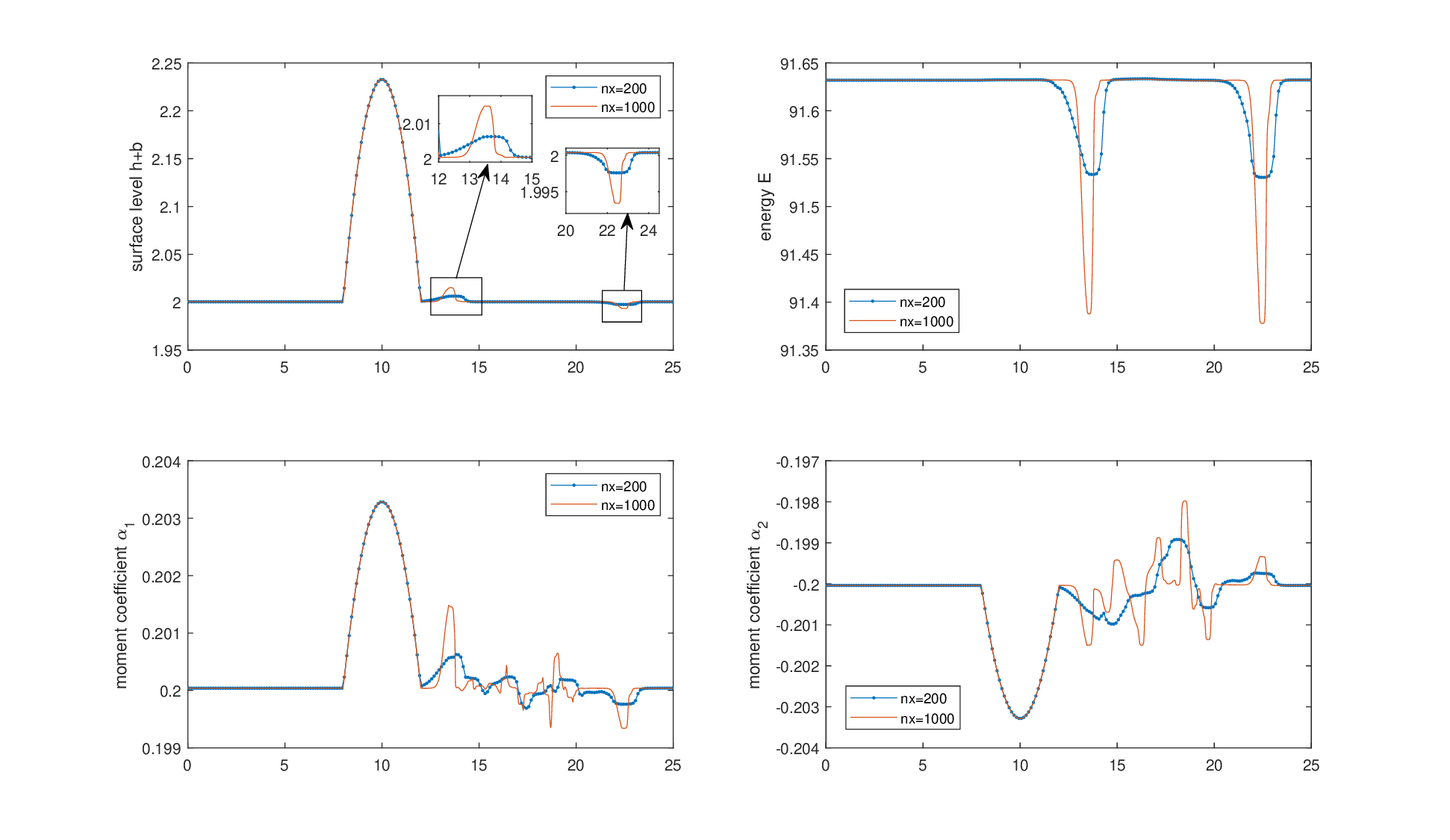}
    \caption{Example \ref{exmp:moving_pert}. Small perturbation of moving water equilibrium for supercritical flow with continuous bottom, using $nx=200$ and $nx=1000$ grid cells.}
    \label{fig:SupCompareC}
\end{figure}

\begin{figure}
    \centering
    \includegraphics[width=1\linewidth]{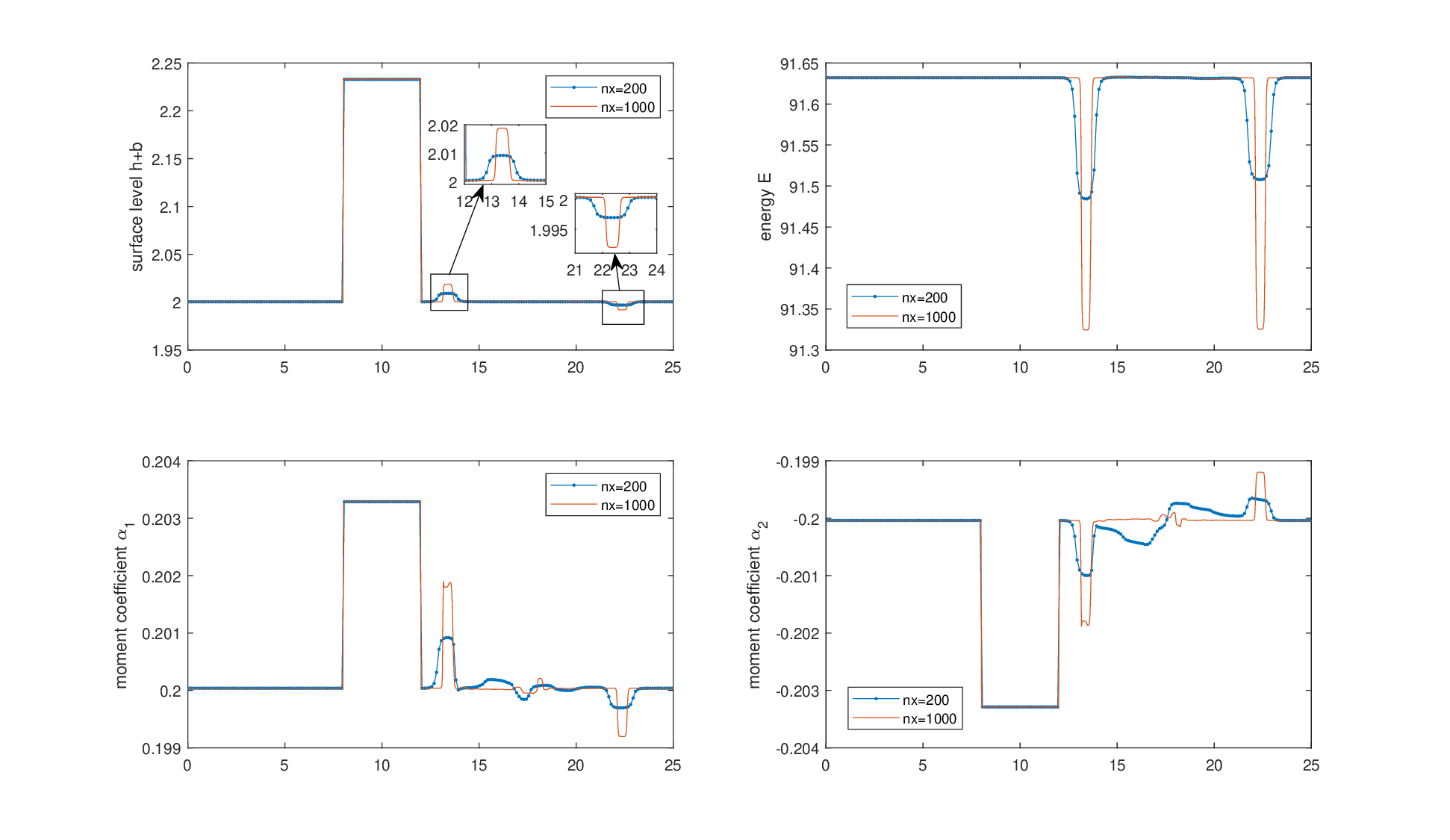}
    \caption{Example \ref{exmp:moving_pert}. Small perturbation of moving water equilibrium for supercritical flow with discontinuous bottom, using $nx=200$ and $nx=1000$ grid cells.}
    \label{fig:SupCompareD}
\end{figure}

%\subsection*{Example 5.6. The standard dam-break test}
\begin{example}\label{exmp:dam}
The standard dam-break test.
\end{example}

We implement a standard dam-break test to validate the correctness of the scheme. Following the example in \cite{koellermeier2022127166}, we consider the initial conditions for the dam-break problem as
\begin{equation}
    h(x,0)=\left\{
    \begin{array}{ll}
    3,& \text{if~}x\geq 0, \\
    1,& \text{if~}x < 0.
    \end{array}\right.
\end{equation}
\begin{equation}
    u_m(x,0)=0.25,\quad \alpha_1(x,0)=-0.25,\quad \alpha_2(x,0)=0.25.
\end{equation}
The water height $h$, average velocity $u_m$, and moment coefficients $\alpha_1,\alpha_2$ are computed at time $t=0.04$. The numerical results obtained using the PCDG-still and PCDG-moving schemes with $nx=400$ uniform grid cells are presented in Figure \ref{fig:DBcompare}. From the figure, it is evident that the PCDG-still and PCDG-moving schemes produce nearly identical results for the height $h$ and average velocity $u_m$. However, for the moment coefficients $\alpha_1$ and $\alpha_2$,  the PCDG-moving scheme demonstrates improved accuracy in capturing the rarefaction wave.
% while the PCDG-still scheme exhibits slight oscillations in the results.
Due to the use of different equilibrium variables when evaluating the non-conservative products, the two schemes produce slightly different moment coefficients.
Although the initial condition is not a steady state or a perturbation around one, the PCDG-moving scheme still provides noticeably better results.
\begin{figure}
\centering
\includegraphics[width=1\linewidth]{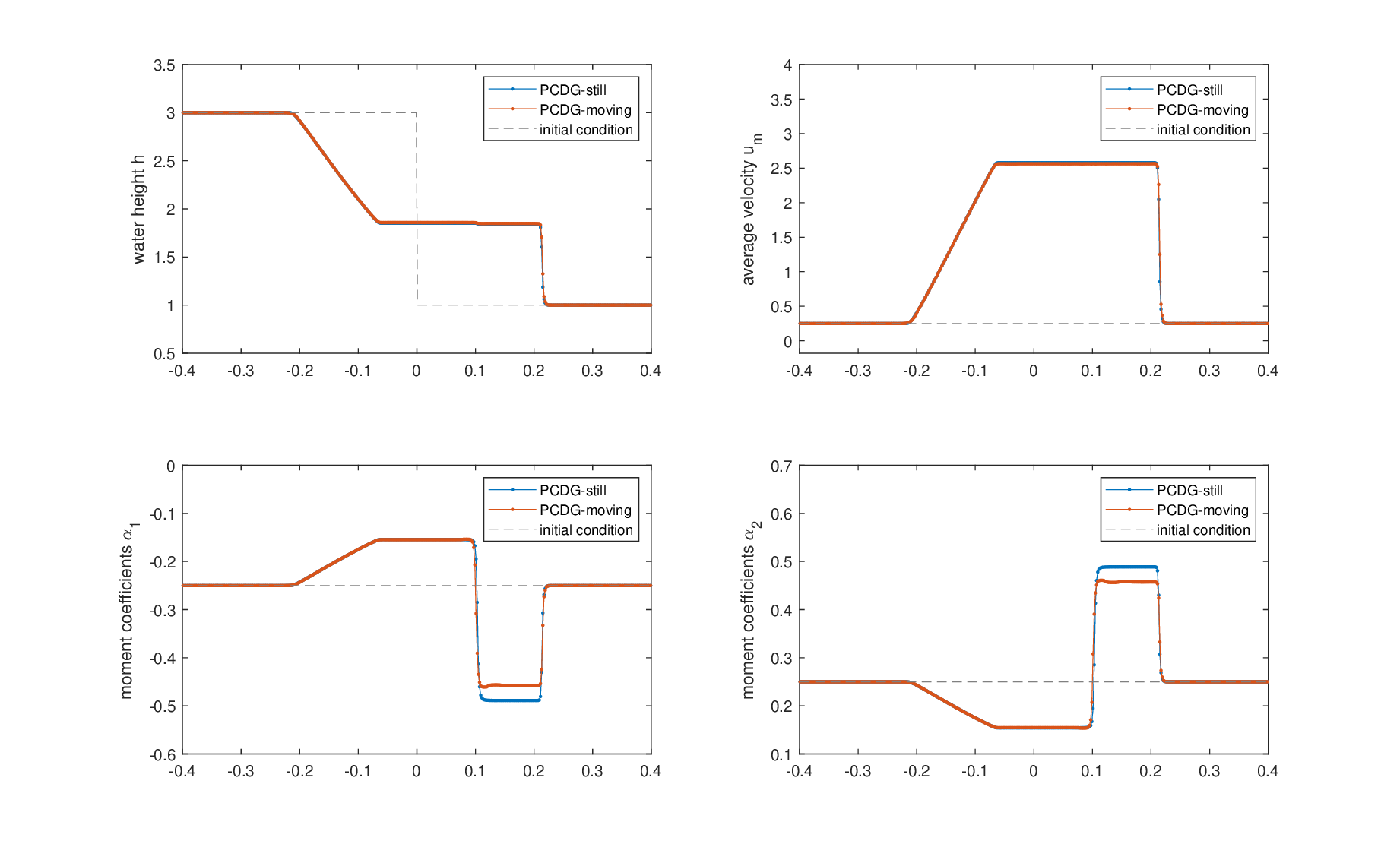}
\caption{\label{fig:DBcompare}Example \ref{exmp:dam}. The standard dam-break test, using PCDG-still and PCDG-moving method with $nx=400$ uniform cells at time $t=0.04$.}
\end{figure}

%\subsection*{Example 5.7. The dam-break test over a discontinuous bottom topography}
\begin{example}\label{exmp:dam_disc}
The dam-break test over a discontinuous bottom topography.
\end{example}
We also implement a dam-break test over a discontinuous bottom topography, similar to the test case in \cite{xing2006new}. This involves rapidly varying flow over a rectangular bump within the domain $[0,1500]$, which is given by
\begin{equation}
    \left.b(x)=\left\{
    \begin{array}{ll}
    8,& \text{if }562.5\leq x\leq937.5, \\
    0,& \text{otherwise}.
    \end{array}\right.\right.
\end{equation}
To examine the differences between the SWLME and the SWE equations, we note that the SWE typically uses a discharge $hu_m = 0$ as the initial condition in this test case. In contrast, we set non-zero initial values for the velocity $u_m$ and the moment coefficients $\alpha_1, \alpha_2$. The initial conditions are specified as follows:
\begin{equation}
    \left.u_m(x,0)=0.25,\quad \alpha_1=-0.25,\quad \alpha_2 = 0.25, \quad
    h(x,0)=\left\{
    \begin{array}{ll}
    20-b(x),& \text{if }x\leq750, \\
    15-b(x),& \text{otherwise}.
    \end{array}\right.\right.
\end{equation}

We first compute the surface level $ h + b $, average velocity $ u_m $, and moment coefficients $ \alpha_1 $, $ \alpha_2 $ at $ t = 15 $ and $ t = 60 $ using $nx=1000$ uniform grid cells. The numerical results are shown in Figures \ref{fig:DBbumpcompare15v2} and \ref{fig:DBbumpcompare60v2}. At $ t = 15 $, both the PCDG-still and PCDG-moving methods yield similar results, effectively simulating the Riemann problem with a discontinuous bottom topography. However, at $ t = 60 $, differences arise between the two methods. The PCDG-still method, which performs less accurately when the initial condition $ hu_m \neq 0 $, accumulates larger errors in the long-term simulation compared to the PCDG-moving method. In contrast, the PCDG-moving method continues to accurately capture the long-term evolution of the Riemann problem.
\begin{figure}
\centering
\includegraphics[width=1\linewidth]{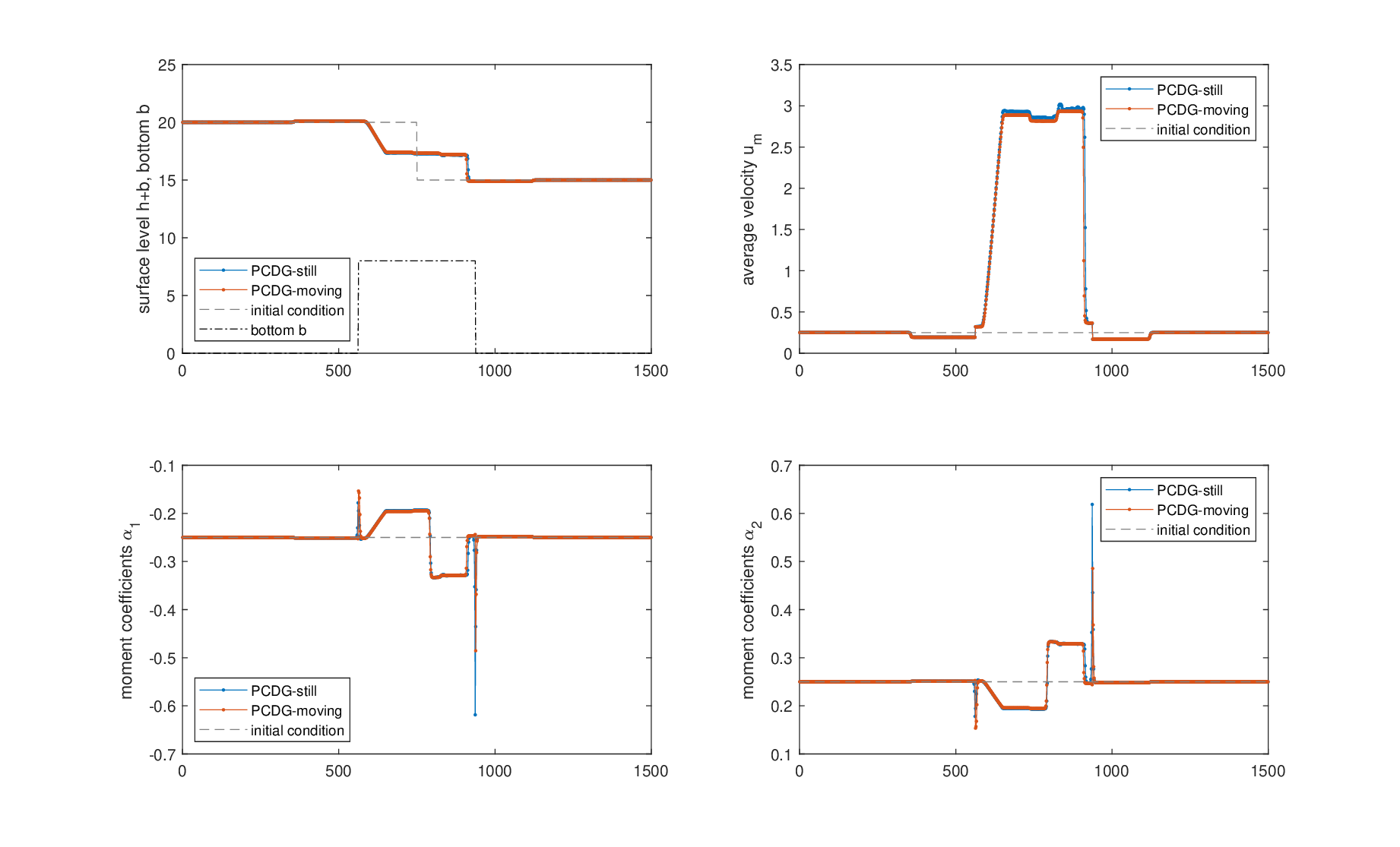}
\caption{\label{fig:DBbumpcompare15v2}Example \ref{exmp:dam_disc}. The dam-break test over a rectangular bump, using PCDG-still and PCDG-moving method with $nx=1000$ uniform cells at time $t = 15$.}
\end{figure}

\begin{figure}
\centering
\includegraphics[width=1\linewidth]{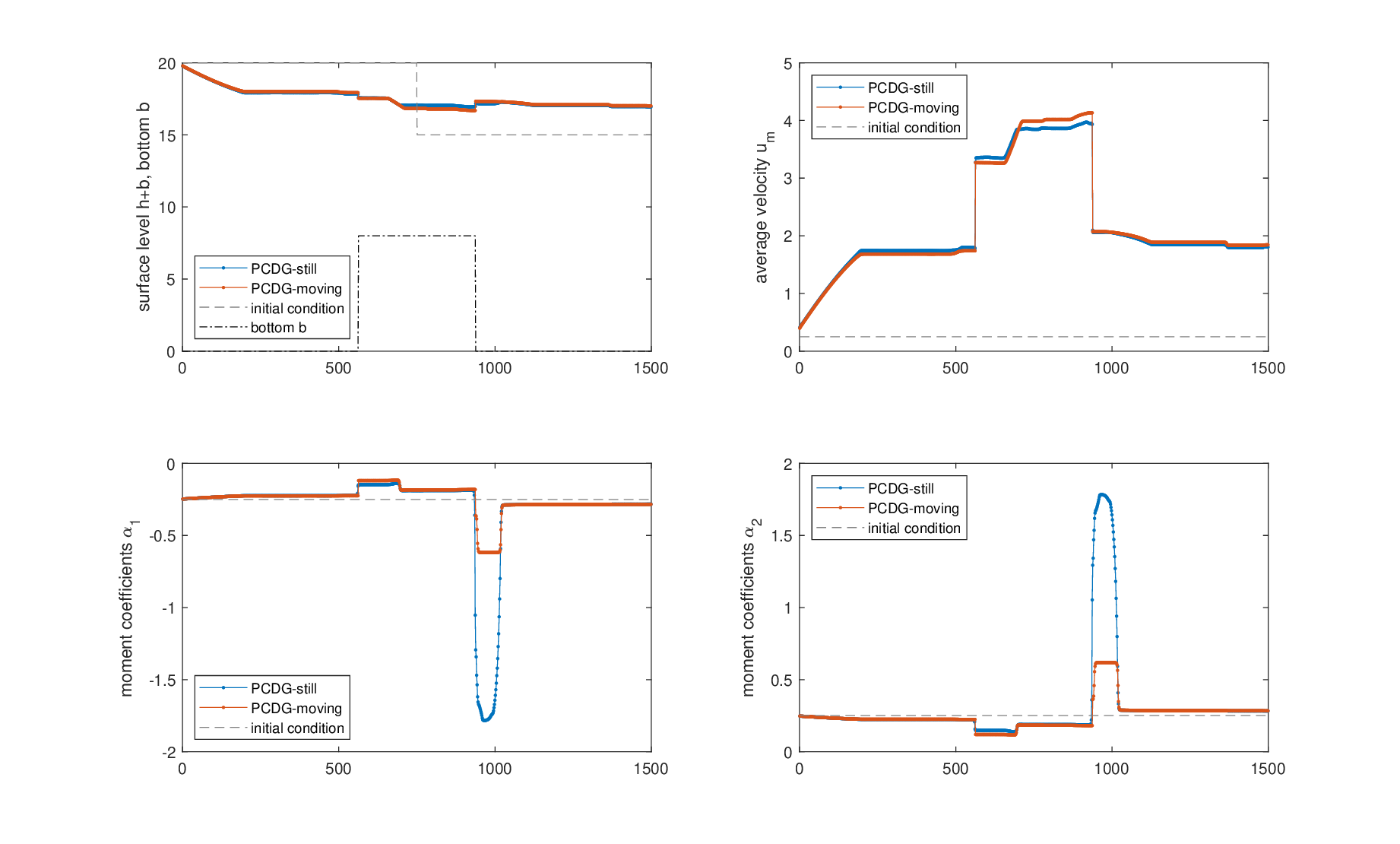}
\caption{\label{fig:DBbumpcompare60v2}Example \ref{exmp:dam_disc}. The dam-break test over a rectangular bump, using PCDG-still and PCDG-moving method with $nx=1000$ uniform cells at time $t = 60$.}
\end{figure}

Next, we compute the numerical solutions at $ t = 15 $ and $ t = 60 $ using both the PCDG-still and PCDG-moving methods with $nx=200$ uniform grid cells. For comparison, we use the PCDG-moving method with $nx=2000$ uniform grid cells as the reference solution. The results are shown in Figure \ref{fig:WBDGcompare}. At $ t = 15 $, both the PCDG-still and PCDG-moving methods produce results that align well with the reference solution. However, at $ t = 60 $, the errors for the PCDG-still method become more pronounced compared to the PCDG-moving method, highlighting the advantage of the latter in handling the long-term evolution of non-zero initial velocity conditions.

\begin{figure}
\centering
\includegraphics[width=1\linewidth]{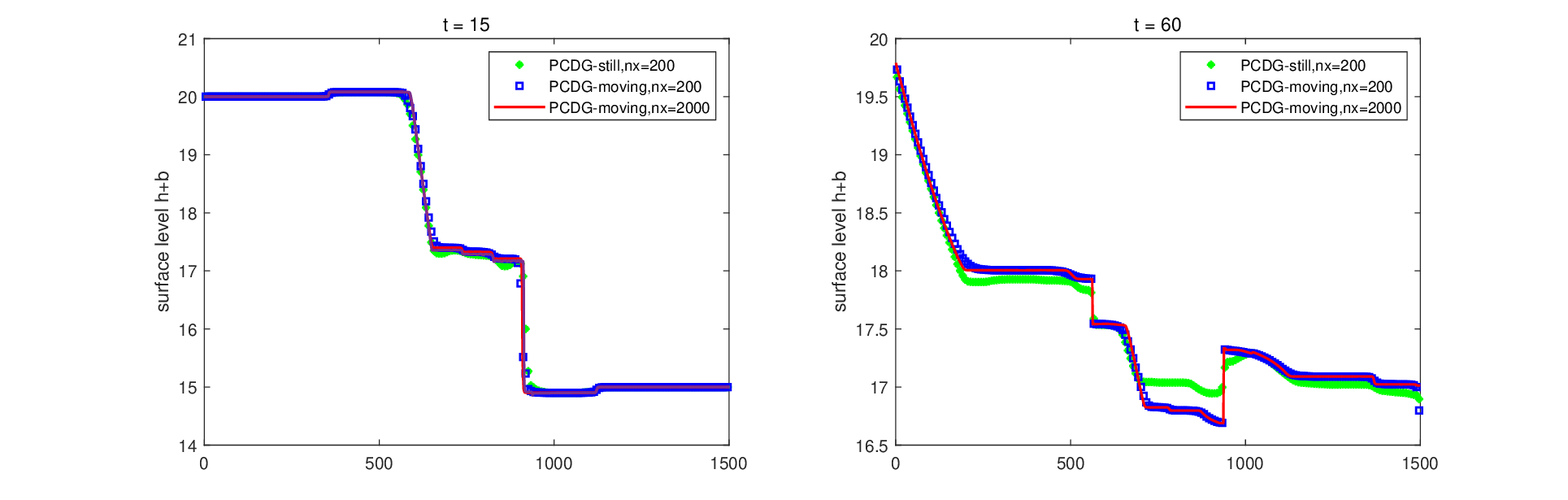}
\caption{\label{fig:WBDGcompare}Example \ref{exmp:dam_disc}. The dam-break test over a rectangular bump, using PCDG-still and PCDG-moving method with $nx=200$ and $nx=1000$ uniform cells at time $t = 15$ (left) and $t = 60$ (right).}
\end{figure}

\begin{example}\label{exmp:dam_square}
The dam-break test with square root velocity profile
\end{example}
This test case is adapted from the transient model comparison presented in \cite{koellermeier2022127166} and is used here to evaluate the performance of the proposed PCDG scheme in handling non-zero moments at a higher moment order $N = 8$. Similar to Example \ref{exmp:dam}, the initial water height follows a standard dam-break configuration:
\begin{equation}
    h(x,0)=\left\{
    \begin{array}{ll}
    3,& \text{if~}x\geq 0, \\
    1,& \text{if~}x < 0,
    \end{array}\right.
    \quad
    u_m(x,0)=0.25.
\end{equation}
The initial vertical velocity profile is prescribed as
\begin{equation}
    u(\zeta) = u_m+\sum_{i=1}^N\alpha_i\phi_i(\zeta)=\sqrt{\zeta},
\end{equation}
which is projected onto the orthonormal basis functions $\phi_i(\zeta)$ to determine the initial moment coefficients $\alpha_i(x,0)$. Following the approach in \cite{koellermeier2022127166}, the projection yields:
\begin{equation}
\begin{aligned}
    & \alpha_1 = -\frac{3}{5}, \quad \alpha_2 = -\frac{1}{7}, \quad \alpha_3 = -\frac{1}{15}, \quad \alpha_4 = -\frac{3}{77}, \\
    & \alpha_5 = -\frac{1}{39}, \quad \alpha_6 = -\frac{1}{55}, \quad \alpha_7 = -\frac{3}{221}, \quad \alpha_8 = -\frac{1}{95}.
\end{aligned}
\end{equation}

Figure \ref{fig:DBsquarecompare} presents the numerical results for the water height $h$, average velocity $u_m$, and moment coefficients $\alpha_1,\dots,\alpha_8$ at time $t = 0.04$. The simulation is conducted using $nx=400$ uniform cells with both the PCDG-still and PCDG-moving schemes. As shown in the figure, the two schemes produce nearly identical results for $h$ and $u_m$ and moment coefficients, demonstrating that both methods retain stability and accuracy even at elevated moment resolution.

\begin{figure}[htb]
\centering
\includegraphics[width=1\linewidth]{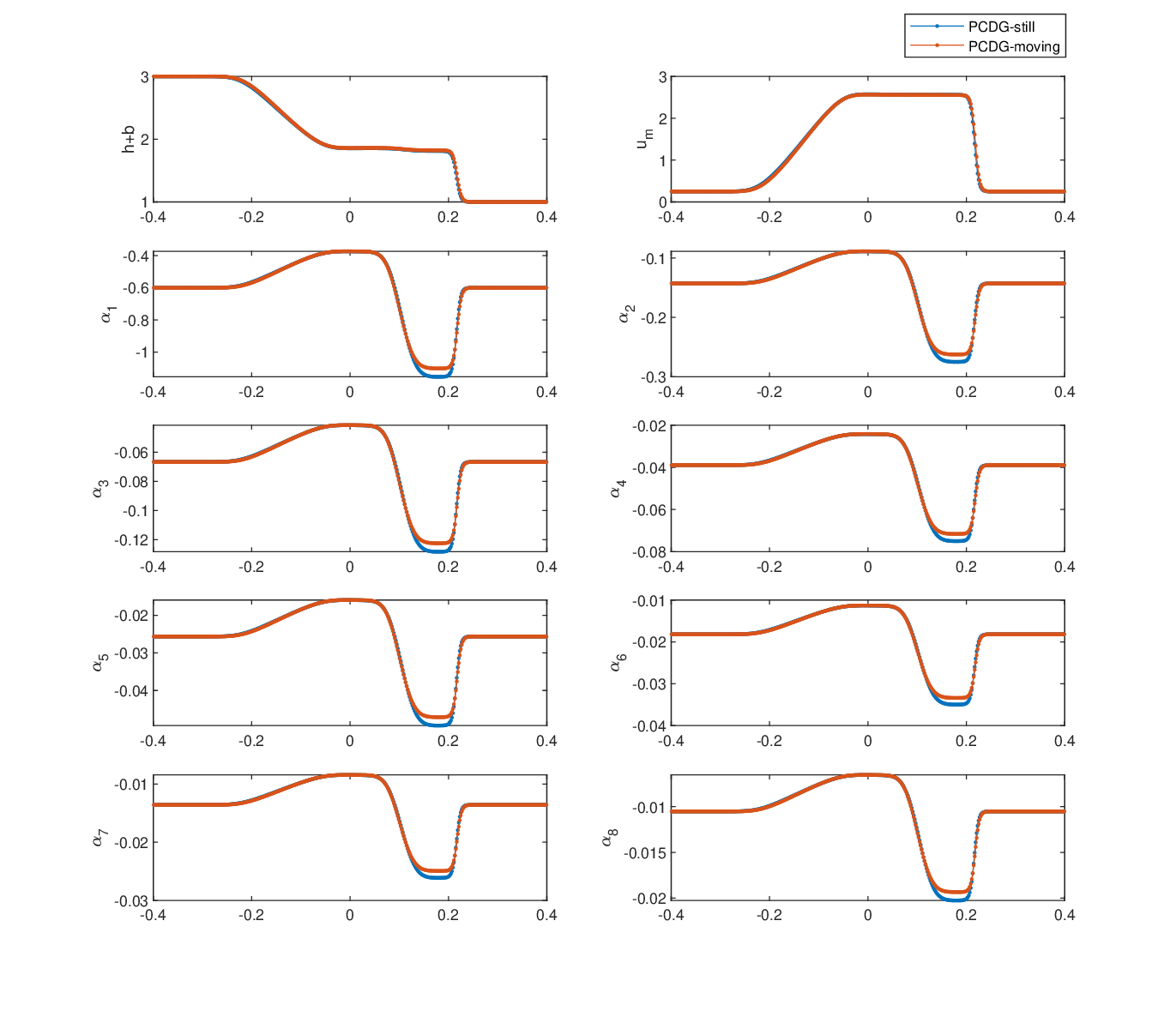}
\caption{\label{fig:DBsquarecompare}Example \ref{exmp:dam_square}. The dam-break test with square root velocity profile, using PCDG-still and PCDG-moving method with $nx=400$ uniform cells at time $t=0.04$.}
\end{figure}

\section{Conclusion}
\label{se-con}
We have developed two high-order well-balanced path-conservative discontinuous Galerkin methods for the shallow water linearized moment equations. For the still water equilibrium, by redefining the conservative variables and reformulating the system into a quasilinear form without explicit source terms, our methods preserve the equilibrium exactly without the need for hydrostatic reconstruction. For the moving water equilibrium, we extend the DG framework to equilibrium-preserving function spaces using linear segment paths, thereby ensuring the well-balanced property. Both theoretical analysis and numerical experiments confirm that the proposed methods maintain exact preservation of equilibrium states while achieving high-order accuracy. The methods also demonstrate robustness and high resolution in capturing complex flow dynamics. We expect that this framework can be generalized to other non-conservative systems with more intricate equilibrium structures. As part of future work, we aim to extend the methodology to more complex models, such as those involving sediment transport.

\section*{Funding}
The research of Julian Koellermeier was partially supported by the project \textit{HiWAVE} with file number VI.Vidi.233.066 of the \textit{ENW Vidi} research programme, funded by the \textit{Dutch Research Council (NWO)} under the grant \url{https://doi.org/10.61686/CBVAB59929}.   The research of Yinhua Xia was partially supported by Anhui Provincial Natural Science Foundation grant No. 2408085J004 and NSFC grant No. 12271498. The research of Yan Xu was partially supported by NSFC grant No. 12471347 and Laoshan Laboratory  (No.LSKJ202300305).

% \begin{figure}
% \centering
% \includegraphics[width=0.25\linewidth]{frog.jpg}
% \caption{\label{fig:frog}This frog was uploaded via the file-tree menu.}
% \end{figure}

\bibliographystyle{abbrv}
\bibliography{sample}

\end{document}